\documentclass[12pt]{amsart}

\usepackage{graphicx}

\newtheorem{theorem}{Theorem}
\newtheorem{lemma}[theorem]{Lemma}

\newtheorem{corollary}[theorem]{Corollary}
\theoremstyle{definition}

\newtheorem{example}{Example}
\theoremstyle{remark}
\newtheorem{remark}{Remark}

\newcommand{\boldx}{{\mathbf x}}

\newcommand{\boldzero}{{\mathbf 0}}

\usepackage{tikz}
\usetikzlibrary{positioning,plotmarks,calc,arrows,shapes,shadows,trees}
\usepackage{pgfplots}
\pgfplotsset{compat=1.9}
\pgfdeclarelayer{background}
\pgfdeclarelayer{foreground}
\pgfsetlayers{background,main,foreground}

\begin{document}


\title[Solving Distributed Systems of Equations]{A Kaczmarz Algorithm for Solving Tree Based Distributed Systems of Equations}
\author{Chinmay Hegde}
\address{Electrical and Computer Engineering, Iowa State University, Ames, IA 50011}
\email{chinmay@iastate.edu}
\author{Fritz Keinert}
\address{Department of Mathematics, Iowa State University, 396 Carver Hall, Ames, IA 50011}
\email{keinert@iastate.edu}
\author{Eric S. Weber}
\address{Department of Mathematics, Iowa State University, 396 Carver Hall, Ames, IA 50011}
\email{esweber@iastate.edu}
\subjclass[2000]{Primary: 65F10, 15A06; Secondary 68W15, 41A65}
\keywords{Kaczmarz method, linear equations, least squares, distributed optimization}
\date{\today}
\begin{abstract}
The Kaczmarz algorithm is an iterative method for solving systems of linear equations.  We introduce a modified Kaczmarz algorithm for solving systems of linear equations in a distributed environment, i.e.~the equations within the system are distributed over multiple nodes within a network.  The modification we introduce is designed for a network with a tree structure that allows for passage of solution estimates between the nodes in the network.  We prove that the modified algorithm converges under no additional assumptions on the equations.  We demonstrate that the algorithm converges to the solution, or the solution of minimal norm, when the system is consistent.  We also demonstrate that in the case of an inconsistent system of equations, the modified relaxed Kaczmarz algorithm converges to a weighted least squares solution as the relaxation parameter approaches $0$.
\end{abstract}
\maketitle




\section{Introduction}

The Kaczmarz method (\cite{K-37}, 1937) is an iterative algorithm for
solving a system of linear equations $A \vec{x} = \vec{b}$, where $A$ is an $m
\times k$ matrix.  Written out, the equations are   $ \vec{a}_i \cdot \vec{x}  = b_i$ for $i=1,\ldots,m$,  where $\vec{a}_i^{T}$ is the $i$th row of the matrix $A$, and we take the dot product to be linear in both variables.
Given a solution guess $\vec{x}^{(n)}$ and an equation number $i$, we calculate
$r_i = b_i -  \vec{a}_i \cdot \vec{x}^{(n)} $ (the residual for equation $i$), and define
\begin{equation}\label{Eq:update}
  \vec{x}^{(n+1)} = \vec{x}^{(n)} + \frac{r_i}{\|\vec{a}_i\|^2} \vec{a}_i.
\end{equation}
This makes the residual of $\vec{x}^{(n+1)}$ in equation $i$ equal to 0. Here and elsewhere, $\| \cdot \|$ is the usual Euclidean ($\ell^{2}$) norm.  We iterate repeatedly through all equations (i.e. we consider $\lim_{n \to \infty} \vec{x}^{(n)}$ where $n+1 \equiv i \mod m$).  Kaczmarz proved that if the system of equations has a unique solution, then $\vec{x}^{(n)}$ converges to that solution.  Later, it was proved in \cite{T-71} that if the system is consistent (but the solution is not unique), then the sequence converges to the solution of minimal norm.  Likewise, it was proved in  \cite{EHL-81,N-86} that if inconsistent, a relaxed version of the algorithm can provide approximations to a weighted least-squares solution.

Obtaining the $n+1$ estimate requires knowledge only of the $i$-th equation ($n+1 \equiv i \mod m$ as above) and the $n$-th estimate.  We suppose that the equations are indexed by the nodes of a tree, representing a network in which the equations are distributed over many nodes.  In our distributed Kaczmarz algorithm, solution estimates can only be communicated when there exists an edge between the nodes.  The estimates for the solution will disperse through the tree, which results in several different estimates of the solution.  When these estimates then reach the leaves of the tree, they are pooled together into a single estimate.  Using this single estimate as a seed, the process is repeated, with the goal that the sequence of single estimates will converge to the true solution.  We illustrate the dispersion and pooling processes in Figure \ref{Fig:eqs}.

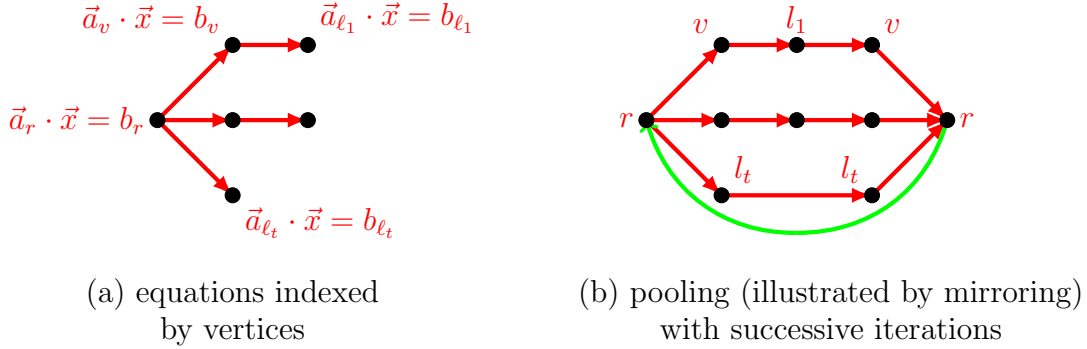
\begin{figure}[h]
  \centering
  \begin{tikzpicture}
    
    \begin{scope}[xshift=1.0cm,yshift=0.0cm]
    \coordinate (Origin)   at (0,0);

   \draw [ultra thick,-latex,red] (0,0) node [left] {$\vec{a}_{r} \cdot \vec{x} = b_{r}$}
        -- (1,1) node [above left] {$\vec{a}_{v} \cdot \vec{x} = b_{v}$ };
   \draw [ultra thick,-latex,red] (0,0) node [above right] {}
        -- (1,0) ;
   \draw [ultra thick,-latex,red] (0,0)
        -- (1,-1) node [below right] {$\vec{a}_{\ell_{t}} \cdot \vec{x} = b_{\ell_{t}}$};
   \draw [ultra thick,-latex,red] (1,1) node [above right] {}
        -- (2,1) node [above right] {$\vec{a}_{\ell_{1}} \cdot \vec{x} = b_{\ell_{1}}$ };
   \draw [ultra thick,-latex,red] (1,0)
        -- (2,0) node [above left] {};

   \node[draw,circle,inner sep=2pt,fill] at (0,0) {};
   \node[draw,circle,inner sep=2pt,fill] at (1,1) {};
   \node[draw,circle,inner sep=2pt,fill] at (1,0) {};
   \node[draw,circle,inner sep=2pt,fill] at (1,-1) {};
   \node[draw,circle,inner sep=2pt,fill] at (2,0) {};
   \node[draw,circle,inner sep=2pt,fill] at (2,1) {};

\end{scope}



\begin{scope}[xshift=7.5cm,yshift=0.0cm]
    \coordinate (Origin)   at (0,0);

   \draw [ultra thick,-latex,red] (0,0) node [left] {$r$}
        -- (1,1) node [above left] {$v$};
   \draw [ultra thick,-latex,red] (0,0) node [above right] {}
        -- (1,0) ;
   \draw [ultra thick,-latex,red] (0,0)
        -- (1,-1) node [above right] {$l_{t}$};
   \draw [ultra thick,-latex,red] (1,1) node [above right] {}
        -- (2,1) node [above] {$l_{1}$} ;
   \draw [ultra thick,-latex,red] (1,0)
        -- (2,0) node [above left] {};
   \draw [ultra thick,-latex,red] (2,1) node [below right] {}
        -- (3,1) node [above right] {$v$};
   \draw [ultra thick,-latex,red] (3,1) node [above right] {}
        -- (4,0) node [right] {$r$} ;
   \draw [ultra thick,-latex,red] (2,0)
        -- (3,0) node [above left] {};
   \draw [ultra thick,-latex,red] (3,0)
        -- (4,0) node [above left] {};
   \draw [ultra thick,-latex,red] (3,-1)
        -- (4,0) node [above left] {};
   \draw [ultra thick,-latex,red] (1,-1)
        -- (3,-1) node [above left] {$l_{t}$};

   \draw[ultra thick,green,-] (4,0) to[bend left=37.5] (2,-1.5);
   \draw[ultra thick,green,->] (2,-1.5) to[bend left=37.5] (0,0);

   \node[draw,circle,inner sep=2pt,fill] at (0,0) {};
   \node[draw,circle,inner sep=2pt,fill] at (1,1) {};
   \node[draw,circle,inner sep=2pt,fill] at (1,0) {};
   \node[draw,circle,inner sep=2pt,fill] at (1,-1) {};
   \node[draw,circle,inner sep=2pt,fill] at (2,0) {};
   \node[draw,circle,inner sep=2pt,fill] at (2,1) {};
   \node[draw,circle,inner sep=2pt,fill] at (3,1) {};
   \node[draw,circle,inner sep=2pt,fill] at (3,0) {};
   \node[draw,circle,inner sep=2pt,fill] at (3,-1) {};
   \node[draw,circle,inner sep=2pt,fill] at (4,0) {};

\end{scope}

\tikzstyle{la}=[anchor=south];
\node [la] (ll) at (2,-3.2) { \begin{tabular}{c} (a) equations indexed  \\  by vertices \end{tabular}};
\node [la,right=2cm of ll] (lm) { \begin{tabular}{c} (b) pooling (illustrated by mirroring) \\ with successive iterations \end{tabular} } ;


 \end{tikzpicture}
  \caption{Illustration of equations indexed by nodes of the tree.}
  \label{Fig:eqs}
\end{figure}

\subsection{Notation}
For linear transformations $T$, we denote by $\mathcal{N}(T)$ and $\mathcal{R}(T)$ the kernel (nullspace) and range, respectively.

As mentioned previously, our notation is that the dot product of two vectors $\vec{x} \cdot \vec{z} = \sum_{k} x_{k} z_{k}$ is linear in both variables.  We use $\langle \cdot, \cdot \rangle$ to denote the inner product on $\mathbb{C}^d$ which is sesquilinear.  In the sequel, we will use the linear transformation notation (rather than dot product notation):
\begin{equation} \label{Eq:linear-eq}
S_{\vec{a}} : \mathbb{C}^{d} \to \mathbb{C} : \vec{z} \mapsto \vec{a} \cdot \vec{z}.
\end{equation}
When the vector $\vec{a} = \vec{a}_{i}$ corresponds to a row of the matrix $A$ indexed by a natural number $i$, or when $\vec{a} = \vec{a}_{v}$ corresponds to a row of the matrix $A$ indexed by a node $v$, we will denote the transformation in Equation (\ref{Eq:linear-eq}) by $S_{i}$ or $S_{v}$, respectively.  We use $P_{v}$ to denote the linear projection onto $\mathcal{N}(S_{v})$:
\begin{equation} \label{Eq:linear-proj}
P_{v}( \vec{z}) =  \left( I - S_{v}^{*} \left(S_{v} S_{v}^{*}\right)^{-1} S_{v} \right) (\vec{z})
\end{equation}
and $Q_{v}$ to denote the affine projection onto the linear manifold $S_{v} (\vec{z}) = b_{v}$:
\begin{equation} \label{Eq:affine-proj}
Q_{v} (\vec{z}) = P_{v} (\vec{z}) + \vec{h}_{v}
\end{equation}
where $\vec{h}_{v}$ is the vector that satisfies $S_{v} (\vec{h}_{v}) = b_{v}$ and is orthogonal to $\mathcal{N}(S_{v})$.  

A tree is a connected graph with no cycles.  We denote arbitrary nodes (vertices) of a tree by $v$, $u$.  Our tree will be rooted; the root of the tree is denoted by $r$.  Following the notation from MATLAB, when $v$ is on the path from $r$ to $u$, we will say that $v$ is a predecessor of $u$ and write $u \prec v$.  Conversely, $u$ is a successor of $v$.  By immediate successor of $v$ we mean a successor $u$ such that there is an edge between $v$ and $u$ (this is referred to as a \emph{child} in graph theory parlance \cite{west1996graph}).  Similarly, $v$ is an immediate predecessor (i.e. \emph{parent}).  We denote the set of all immediate successors of node $v$ by $\mathcal{C}(v)$.  A node without a successor is called a leaf; leaves of the tree are denoted by $\ell$.  We will denote the set of all leaves by $\mathcal{L}$.  Often we will have need to enumerate the leaves as $\ell_{1}, \dots, \ell_{t}$, hence $t$ denotes the number of leaves.  


A weight $w$ is a nonnegative function on the edges of the tree;  we denote this by $w(u,v)$, where $u$ and $v$ are nodes that have an edge between them.  We assume $w(u,v) = w(v,u)$, though we will typically write $w(u,v)$ when $u \prec v$.  When $u \prec v$, but $u$ is not a immediate successor, we write
\begin{equation} \label{Eq:weight-path}
w(u,v) := \prod_{j=1}^{J-1} w(u_{j+1}, u_{j})
\end{equation}
where $u = u_{1} , \dots, u_{J} = v$ is a path from $u$ to $v$.

When the system of equations $A \vec{x} = \vec{b}$ has a unique solution, we will denote this by $\vec{x}^{S}$.  When the system is consistent but the solution is not unique, we denote the solution of minimal norm by $\vec{x}^{M}$, which is given by
\begin{equation} \label{Eq:soln-min-norm}
 \vec{x}^{M} = \text{argmin } \{ \| \vec{x} \| : A \vec{x} = \vec{b} \}.
\end{equation}

\subsection{The Distributed Kaczmarz Algorithm}
The iteration begins with an estimate, say $\vec{x}^{(n)}$ at the root of the tree (we denote this by $\vec{x}^{(n)}_{r}$).  Each node $u$ receives  from its immediate predecessor $v$ an input estimate $\vec{x}^{(n)}_{v}$ and generates a new estimate via the Kaczmarz update:
\begin{equation} \label{Eq:tree-update}
 \vec{x}^{(n)}_{u} = \vec{x}^{(n)}_{v} + \dfrac{r_{u} (\vec{x}^{(n)}_{v}) }{\| \vec{a}_{u} \|^2 } \vec{a}_{u},
\end{equation}
where the residual is given by 
\begin{equation} \label{Eq:residual}
r_{u}(\vec{x}^{(n)}_{v}) := b_{u} - S_{u} \vec{x}_{v}^{(n)}.
\end{equation} 
Node $u$ then passes this estimate to all of its immediate successors, and the process is repeated recursively.   We refer to this as the \emph{dispersion stage}.   Once this process has finished, each leaf $\ell$ of the tree now possesses an estimate: $\vec{x}^{(n)}_{\ell}$.

The next stage, which we refer to as the \emph{pooling stage}, proceeds as follows.  For each leaf, set $\vec{y}^{(n)}_{\ell} = \vec{x}^{(n)}_{\ell}$.  Each node  $v$ receives as input the several estimates $y^{(n)}_{u}$ from all immediate successors $u$, and calculates an updated estimate as:
\begin{equation} \label{Eq:simple-back}
\vec{y}^{(n)}_{v} = \sum_{u \in \mathcal{C}(v)} w(u,v) \vec{y}^{(n)}_{u},
\end{equation}
subject to the constraints that $w(u,v) > 0$ when $u \in \mathcal{C}(v)$ and $\sum_{u \in \mathcal{C}(v) } w(u,v) = 1$.  This process continues until reaching the root of the tree, resulting in the estimate $\vec{y}^{(n)}_{r}$.

We set $\vec{x}^{(n+1)} = \vec{y}^{(n)}_{r}$, and repeat the iteration.  The updates in the dispersion stage (Equation \ref{Eq:tree-update}) and pooling stage (Equation \ref{Eq:simple-back}) are illustrated in Figure \ref{Fig:updates}.

\begin{figure}[h]
  \centering
  \begin{tikzpicture}
    \coordinate (Origin)   at (0,0);

   \draw [ultra thick,-latex,red] (0,0) node [left] {$\vec{x}^{(n)}_{r}$}
        -- (1,1) node [above left]  {$\vec{x}^{(n)}_{v}$ };
   \draw [ultra thick,-latex,red] (0,0) node [above right] {}
        -- (1,0) ;
   \draw [ultra thick,-latex,red] (0,0)
        -- (1,-1) node [below left] {$\vec{x}^{(n)}_{\ell_{t}}$};
   \draw [ultra thick,-latex,red] (1,1) node [above right] {}
        -- (2,1) node [above right] {$\vec{x}^{(n)}_{\ell_{1}}$ };
   \draw [ultra thick,-latex,red] (1,0)
        -- (2,0) node [above left] {};

   \node[draw,circle,inner sep=2pt,fill] at (0,0) {};
   \node[draw,circle,inner sep=2pt,fill] at (1,1) {};
   \node[draw,circle,inner sep=2pt,fill] at (1,0) {};
   \node[draw,circle,inner sep=2pt,fill] at (1,-1) {};
   \node[draw,circle,inner sep=2pt,fill] at (2,0) {};
   \node[draw,circle,inner sep=2pt,fill] at (2,1) {};



\begin{scope}[xshift=7.5cm,yshift=0.0cm]
    \coordinate (Origin)   at (0,0);

   \draw [ultra thick,-latex,red] (0,0) node [below right] {}
        -- (1,1) node [above right] {};
   \draw [ultra thick,-latex,red] (0,0) node [above right] {}
        -- (1,0) ;
   \draw [ultra thick,-latex,red] (0,0)
        -- (1,-1) node [above left] {};
   \draw [ultra thick,-latex,red] (1,1) node [above right] {}
        -- (2,1) ;
   \draw [ultra thick,-latex,red] (1,0)
        -- (2,0) node [above left] {};
   \draw [ultra thick,-latex,red] (2,1) node [below right] {}
        -- (3,1) node [above right] {};
   \draw [ultra thick,-latex,red] (3,1) node [above right] {$\vec{y}^{(n)}_{v}$}
        -- (4,0) ;
   \draw [ultra thick,-latex,red] (2,0)
        -- (3,0) node [above left] {};
   \draw [ultra thick,-latex,red] (3,0)
        -- (4,0) node [right] {$\vec{y}^{(n)}_{r}$};
   \draw [ultra thick,-latex,red] (3,-1) node [above left] {$\vec{y}^{(n)}_{\ell_{t}}$}
        -- (4,0) node [above left] {};
   \draw [ultra thick,-latex,red] (1,-1)
        -- (3,-1) node [above left] {};

   \draw[ultra thick,green,-] (4,0) to[bend left=37.5] (2,-1.5);
   \draw[ultra thick,green,->] (2,-1.5) to[bend left=37.5] (0,0) node [left] {$\vec{x}^{(n+1)}_{r}$};

   \node[draw,circle,inner sep=2pt,fill] at (0,0) {};
   \node[draw,circle,inner sep=2pt,fill] at (1,1) {};
   \node[draw,circle,inner sep=2pt,fill] at (1,0) {};
   \node[draw,circle,inner sep=2pt,fill] at (1,-1) {};
   \node[draw,circle,inner sep=2pt,fill] at (2,0) {};
   \node[draw,circle,inner sep=2pt,fill] at (2,1) {};
   \node[draw,circle,inner sep=2pt,fill] at (3,1) {};
   \node[draw,circle,inner sep=2pt,fill] at (3,0) {};
   \node[draw,circle,inner sep=2pt,fill] at (3,-1) {};
   \node[draw,circle,inner sep=2pt,fill] at (4,0) {};

\end{scope}

\tikzstyle{la}=[anchor=south];
\node [la] (ll) at (1.0,-3.2) { \begin{tabular}{c} (a) updates disperse \\ through nodes  \end{tabular}};
\node [la,right=4cm of ll] (lm) { \begin{tabular}{c} (b) updates pool and \\ \ pass to next iteration \end{tabular} } ;


 \end{tikzpicture}
  \caption{Illustration of updates in the distributed Kaczmarz algorithm with measurements indexed by nodes of the tree.}
  \label{Fig:updates}
\end{figure}
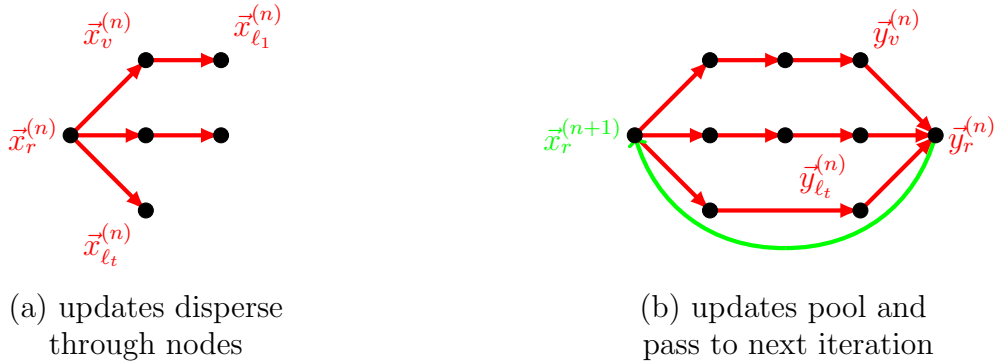

\subsection{Related Work}

The Kaczmarz method was originally introduced in (\cite{K-37}, 1937). It became popular with the introduction of Computer Tomography, under the name of ART (Algebraic Reconstruction Technique).  ART added non-negativity and other constraints to the standard algorithm~\cite{GBH-70}. Other variations on the Kaczmarz method allowed for relaxation parameters \cite{T-71}, re-ordering equations to speed up convergence~\cite{HS-78}, or considering block versions of the Kaczmarz method with relaxation matrices $\Omega_i$~\cite{EHL-81}.   Relatively recently, choosing the next equation randomly has been shown to dramatically improve the rate of convergence of the algorithm \cite{SV-09a,NT14a,NZZ15a}.  Moreover, this randomized version of the Kaczmarz algorithm has been shown to comparable to the gradient descent method \cite{NSW16a}.  Our version of the Kaczmarz method differs from these in that the next equation cannot be chosen randomly or otherwise, since the ordering of the equations is determined \emph{a priori} by the network topology.

Our version is motivated by the situation in which the equations (or measurements) are distributed over a network.  Distributed estimation problems have a long history in applied mathematics, control theory, and machine learning.   At a high level, similar to our approach, they all involve averaging local copies of the unknown parameter vector interleaved with update steps~\cite{tsitsiklis1986distributed,xiao2007distributed,Shah-2008,boyd2011distributed,nedic2009distributed,johansson2009randomized,yuan2016convergence,sayed2014adaptation,zhang2018compressed,scaman2018optimal}.  One common form of the parameter estimation problem involves posing it as a \emph{consensus} problem, where the goal is for nodes in a given graph to arrive at a common solution, assuming that no exchange of measurements takes place and only estimates are shared across neighbors. Computations are
often not synchronized, and network connections may be unstable. Computations done with gossip methods are usually quite simple, such as computing averages, and converge only slowly. 

Following~\cite{yuan2016convergence}, a consensus problem takes the following form.  Consider the problem of minimizing:
\begin{equation*}
F(\vec{x}) = \sum_{v=1}^m f_v(\vec{x}),
\end{equation*}
where $f_v$ is a function that is known (and private) to node $v$ in the graph. Then, one can solve this minimization problem using decentralized gradient descent, where each node updates its estimate of $\vec{x}$ (say $\vec{x}_{v}$) by combining the average of its neighbors with the negative gradient of its local function $f_v$:
\begin{equation*}
\vec{x}_{v}^{(n+1)} = \frac{1}{\text{ deg } v }\sum_{u} m(v,u) \vec{x}_{u}^{(n)} - \omega \nabla f_v (\vec{x}_{u}^{(n)}) ,
\end{equation*}
where $M = (m(v,u)) \in \{0,1\}^{m \times m}$ represents the adjacency matrix of the graph. Specializing $f_v (\vec{x}) = c_{v} (b_{v} - \vec{a}_{v} \cdot \vec{x})^2$ yields our least-squares estimation problem that we establish in Theorem \ref{Th:omega-ic} (where $c_{v}$ is a fixed weight for each node).

However, our version of the Kaczmarz method differs from previous work in a few aspects: (i) we assume a specific (tree) topology; (ii) our updates are asynchronous (the update time for each node is a function of its distance from the root); and (iii) as we will emphasize in Theorem \ref{Th:omega-ic}, we make no strong convexity assumptions.

On the other end of the spectrum are algorithms that distribute a computational task over many processors arranged in a fixed network. These algorithms are usually considered in the context of parallel processing, where the nodes of the graph represent CPUs in a highly parallelized computer. This setup can handle large computational tasks, but the problem must be amenable to being broken into independent pieces. See \cite{BT-97} for an overview.

The algorithm we are considering does not really fit either of those categories. It requires more structure than the gossip algorithms, but each node depends on results from other nodes, more than the usual distributed algorithms.

This was pointed out in \cite{BT-97}. For iteratively solving a system of linear equations, an SOR variant of the Jacobi method is easy to parallelize; standard SOR, which is a variation on Gauss-Seidel, is not. The authors also consider what they call the {\em Reynolds method}, which is similar to a Kaczmarz method with all equations being updated simultaneously. Again, this method is easy to parallelize. A sequential version called RGS (Reynolds Gauss-Seidel) can only be parallelized in certain settings, such as the numerical solution of PDEs.

A distributed version of the Kaczmarz algorithm was introduced in \cite{kamath2015distributed}.  The main ideas presented there are very similar to ours:  updated estimates are obtained from prior estimates using the Kaczmarz update with the equations that are available at the node, and distributed estimates are averaged together at a single node (which the authors refer to as a fusion center, for us it is the root of the tree).  In \cite{kamath2015distributed}, the convergence analysis is limited to the case of consistent systems of equations, and inconsistent systems are handled by Tikhonov regularization \cite{HHLL79a,Hansen2010} rather than by varying the relaxation parameter.

Finally, the Kaczmarz algorithm has been proposed for online processing of data in  \cite{HLH80a,chi2016kaczmarz}.   In these papers, the processing is online, so neither distributed nor parallel.

\section{Analysis of the Kaczmarz Algorithm for Tree Based Distributed Systems of Equations}

In this section, we will demonstrate that the Kaczmarz algorithm for tree based equations as defined in Equations (\ref{Eq:tree-update}) and (\ref{Eq:simple-back}) converges.  We consider three cases separately: (i) the system is consistent and the solution is unique; (ii) the system is consistent but there are many solutions; and (iii) the system is inconsistent.  In subsection \ref{ssec:unique}, we prove that for case (i) the algorithm converges to the solution, and in subsection \ref{ssec:consistent}, we prove that for case (ii) the algorithm converges to the solution of minimal norm.  Also in subsection \ref{ssec:consistent}, we introduce the relaxed version of the update in Equation (\ref{Eq:tree-update}).  We prove that for every relaxation parameter $\omega \in (0,2)$, the algorithm converges to the solution of minimal norm.  Then in subsection \ref{ssec:inconsistent}, we prove that for case (iii) the algorithm converges to a generalized solution $\vec{x}(\omega)$ which depends on $\omega$, and $\vec{x}(\omega)$ converges to a weighted least-squares solution as $\omega \to 0$.

\subsection{Systems with Unique Solutions} \label{ssec:unique}


For our analysis, we need to trace the estimates through the tree.  Suppose that the tree has $t$ leaves; for each leaf $\ell$, let $p_{\ell} - 1$ denote the length of the path between the root $r$ and the leaf $\ell$.  We will denote the vertices on the path from $r$ to $\ell$ by $r = (\ell, 1), (\ell, 2) , \dots, (\ell, p_{\ell}) = \ell$.  During the dispersion stage, we have for $p = 2, \dots , p_{\ell}$:
\[ \vec{x}^{(n)}_{\ell, p} = \vec{x}^{(n)}_{\ell, p-1} + \left( \dfrac{ r_{\ell, p} ( \vec{x}^{(n)}_{\ell, p-1} ) }{ \| \vec{a}_{\ell, p} \|^2 }  \right)\vec{a}_{\ell, p}. \]

Then at the beginning of the pooling stage, we have the estimates $\vec{y}^{(n)}_{\ell}  := \vec{x}^{(n)}_{\ell}$ (we denote $\vec{x}^{(n)}_{\ell} := \vec{x}^{(n)}_{\ell, p_{\ell}}$ and $\vec{y}^{(n)}_{\ell} := \vec{y}^{(n)}_{\ell, p_{\ell}}$).  These estimates then pool back at the root as follows (the proof is a straightforward induction argument):
\begin{lemma} \label{L:pooling}
The estimate at the root at the end of the pooling stage is given by:
\begin{equation*}
\vec{y}^{(n)}_{r} = \sum_{\ell \in \mathcal{L} } w(\ell, r) \vec{y}^{(n)}_{\ell}.
\end{equation*}
\end{lemma}

Note that also by induction, we have that 
\begin{equation} \label{Eq:total-weights}
\sum_{\ell \in \mathcal{L} } w(\ell, r) = 1.
\end{equation}

\begin{theorem} \label{Th:unique-soln}
Suppose that the equation $A \vec{x} = \vec{b}$ has a unique solution, denoted by $\vec{x}^{S}$.  There exists a constant $\alpha < 1$, such that
\[ \| \vec{x}^{S} - \vec{x}^{(n+1)} \| \leq \alpha \| \vec{x}^{S} - \vec{x}^{(n)} \|. \]
Consequently,
\[ \lim_{n \to \infty} \vec{x}^{(n)} = \vec{x}^{S}, \]
and the convergence is linear in order.
\end{theorem}

\begin{proof}
Along any path from the root $r$ to the leaf $\ell$, the dispersion stage is identical to the classical Kaczmarz algorithm, and so we can write (see \cite{KwMy01}):
\begin{equation*}
\vec{x}^{S} - \vec{x}^{(n)}_{\ell} =  P_{\ell, {p_{\ell}}}( \vec{x}^{S} - \vec{x}^{(n)}_{\ell, {p_{\ell} - 1}} )
=  P_{\ell, {p_{\ell}}} \dots  P_{\ell, 2} P_{\ell, 1} ( \vec{x}^{S} - \vec{x}^{(n)} ), \label{Eq:projections}
\end{equation*}
from which it follows immediately that
\begin{equation} \label{Eq:leaf-estimate}
\| \vec{x}^{S} - \vec{x}^{(n)}_{\ell} \| \leq \| \vec{x}^{S} - \vec{x}^{(n)} \|.
\end{equation}

We claim that unless $\vec{x}^{S} = \vec{x}^{(n)}$, we must have a strict inequality for at least one leaf, say $\ell_{0}$.  Indeed, suppose to the contrary that for every leaf $\ell$, we had equality in Equation (\ref{Eq:leaf-estimate}), then by Equation (\ref{Eq:projections}), we must have for every vertex $v = (
\ell, k)$ in the path from the root $r$ to the leaf $\ell$:
\begin{equation} \label{Eq:orthogonal1}
 P_{v} ( \vec{x}^{S} - \vec{x}^{(n)}) = \vec{x}^{S} - \vec{x}^{(n)}.
\end{equation}
Therefore, we obtain  
\begin{equation} \label{Eq:orthogonal2}
S_{v} (\vec{x}^{S} - \vec{x}^{(n)}) = 0  \text{ for all vertices } v.
\end{equation}
By our assumption that the equation has a unique solution, we obtain that $\vec{x}^{S} - \vec{x}^{(n)} = 0$.

By Equations (\ref{Eq:total-weights}) and (\ref{Eq:leaf-estimate}) and our previous claim, we have
\begin{equation} \label{Eq:strict-inequality}
\| \vec{x}^{S} - \vec{x}^{(n+1)} \|  < \sum_{\ell \in \mathcal{L} } w(\ell, r) \| \vec{x}^{S} - \vec{x}^{(n)} \| = \| \vec{x}^{S} - \vec{x}^{(n)} \|.
\end{equation}

By continuity and compactness, there is a uniform constant $\alpha$ less than 1 that satisfies the claim.  This completes the proof.
\end{proof}

As we shall see in the sequel, we can interpret the above proof in the following way:  define the mapping 
\begin{equation*} 
\mathcal{P} := \sum_{\ell \in \mathcal{L}} w(\ell, r)  P_{\ell, {p_{\ell}}} \dots  P_{\ell, 2} P_{\ell, 1},
\end{equation*}
then the mapping $\vec{z} \mapsto \vec{x}^{S} - \mathcal{P}( \vec{x}^{S} - \vec{z} )$ is a contraction with unique fixed point $\vec{x}^{S}$.  Moreover, the iteration of the algorithm can be expressed as:
\begin{equation} \label{Eq:fixedpt1}
\vec{x}^{(n+1)} =  \vec{x}^{S} - \mathcal{P}(\vec{x}^{S} - \vec{x}^{(n)} ).
\end{equation}

\subsection{Consistent Systems} \label{ssec:consistent}

We shall show in this section that the distributed Kaczmarz algorithm as defined in Equations (\ref{Eq:tree-update}) and (\ref{Eq:simple-back}) will converge to the solution with minimal norm in the case that there exists more than one solution.  We first introduce the relaxed version of the algorithm; we will show that for any appropriate relaxation parameter, the relaxed algorithm will converge to the solution of minimal norm.

The relaxed distributed Kaczmarz algorithm for tree based equations is as follows.  Choose a relaxation parameter $\omega > 0$ (generally, we will require $\omega \in (0,2)$, though see Section \ref{S:examples} for further discussion).  At each node $w$ during the dispersion stage of iteration $n$, the update becomes:
\begin{equation} \label{Eq:relaxed-update}
\vec{x}_{w}^{(n)} = \vec{x}_{v}^{(n)} + \omega \dfrac{r_{w}(\vec{x}_{v}^{(n)})}{\| \vec{a}_{w} \|^2} \vec{a}_{w}.
\end{equation}
We suppress the dependence of $\vec{x}_{v}^{(n)}$ on $\omega$, but we will consider the limit
\begin{equation} \label{Eq:x-omega-limit}
 \lim_{n \to \infty} \vec{x}^{(n)} := \vec{x}(\omega)
\end{equation}
which (in general) depends on $\omega$.  We will prove in Theorem \ref{Th:omega-c} that when the system of equations is consistent, then this limit exists and is in fact independent of $\omega$.  We will prove in Theorem \ref{Th:omega-ic} that when the system of equations is inconsistent, then the limit exists, depends on $\omega$, and $\vec{x}(\omega) \to \vec{x}^{LS}$ as $\omega \to 0$, where $\vec{x}^{LS}$ is a weighted least-squares solution.

As in Equations (\ref{Eq:linear-proj}) and (\ref{Eq:affine-proj}), we use $P_{v}$ and $Q_{v}$ to denote the linear and affine projections, respectively.  We will need the fact that $Q_{v}$ is Lipschitz with constant $1$:
\[ \| Q_{v} \vec{z}_{1} - Q_{v} \vec{z}_{2} \| \leq \| \vec{z}_{1} - \vec{z}_{2} \| . \]

The relaxed Kaczmarz update in Equation (\ref{Eq:relaxed-update}) can be expressed as:
\[ \vec{x}_{w}^{(n)} = [(1 - \omega) I + \omega Q_{w}] \vec{x}_{v}^{(n)} =: Q_{w}^{\omega} \vec{x}_{v}^{(n)}. \]
Thus, the estimate $\vec{x}^{(n)}_{\ell}$ of the solution at leaf $\ell$, given the solution estimate $\vec{x}^{(n)}$ as input at the root $r$, is:
\begin{equation} \label{Eq:relaxed-leaf}
\vec{x}^{(n)}_{\ell} = Q_{\ell, p_{\ell}}^{\omega} \cdots Q_{\ell, 2}^{\omega} Q_{\ell, 1}^{\omega} \vec{x}^{(n)} =: \mathcal{Q}_{\ell}^{\omega} \vec{x}^{(n)}.
\end{equation}
We can now write the full update, with both dispersion and pooling stages, of the relaxed Kaczmarz algorithm as:
\begin{equation} \label{Eq:full-update}
\vec{x}^{(n+1)} = \sum_{\ell \in \mathcal{L} } w(\ell, r) \mathcal{Q}_{\ell}^{\omega} \vec{x}^{(n)} =: \mathcal{Q}^{\omega} \vec{x}^{(n)}.
\end{equation}

We note that, as above, each $Q_{v}^{\omega}$ is a Lipschitz map with constant $1$ whenever $0 < \omega < 1$, but in fact, since $Q_{v} \vec{z}_{1} - Q_{v} \vec{z}_{2} = P_{v}\vec{z}_{1} - P_{v}\vec{z}_{2}$, we have that $Q_{v}^{\omega}$ is Lipschitz with constant $1$ whenever $0 < \omega < 2$.  Moreover, as $\sum_{\ell \in \mathcal{L}}  w(\ell, r) = 1$, we obtain: 
\begin{lemma} \label{L:Q-Lip}
For $0 < \omega < 2$, $\mathcal{Q}_{\ell}^{\omega}$ and $\mathcal{Q}^{\omega}$ are Lipschitz with constant $1$.
\end{lemma}

We note that the mappings $Q_{(\cdot)}^{(\cdot)}, \mathcal{Q}_{(\cdot)}^{(\cdot)}$ are affine transformations; we also have use for the analogous linear transformations.  Similar to Equations (\ref{Eq:relaxed-leaf}) and (\ref{Eq:full-update}), we write 
\begin{align*}
P_{v}^{\omega} &:= (1 - \omega) I + \omega P_{v}; \\
\mathcal{P}_{\ell}^{\omega} &:= P_{\ell, p_{\ell}}^{\omega} \cdots P_{\ell, 2}^{\omega} P_{\ell, 1}^{\omega}; \\
\mathcal{P}^{\omega} &:= \sum_{\ell \in \mathcal{L} } w(\ell, r) \mathcal{P}_{\ell}^{\omega}.
\end{align*}

\begin{theorem} \label{Th:omega-c}
If the system of equations given by $A \vec{x} = \vec{b}$ is consistent, then for any $0 < \omega < 2$, the sequence of estimates $\vec{x}^{(n)}$ as given in Equation (\ref{Eq:full-update}) converges to the solution $\vec{x}^{M}$ of minimal norm as given by (\ref{Eq:soln-min-norm}), provided the initial estimate $\vec{x}^{(0)} \in \mathcal{R}(A^{*})$.
\end{theorem}
We shall prove Theorem \ref{Th:omega-c} using a sequence of lemmas.  We follow the argument as presented in Natterer \cite{N-86}, adapting the lemmas as necessary.  For completeness, we will state (without proof) the lemmas that we will use unaltered from \cite{N-86}.  (See also Yosida \cite{Yos68a}.)

\begin{lemma}[\cite{N-86}, Lemma V.3.1]  \label{L:Natt-V.3.1}
Let $T$ be a linear map on a Hilbert space $H$ with $\| T \| \leq 1$.  Then,
\[ H = \mathcal{N}(I - T) \oplus \overline{ \mathcal{R} ( I - T ) }.  \]
\end{lemma}
%


\begin{lemma}[\cite{N-86}, Lemma V.3.2]  \label{L:Natt-3.2}
Suppose $\{ \vec{z}_{k} \}$ is a sequence in $\mathbb{C}^{d}$ such that for any leaf $\ell \in \mathcal{L}$,
\[ \| \vec{z}_{k} \| \leq 1 \text{ and } \lim_{k \to \infty} \|  \mathcal{P}_{\ell}^{\omega} \vec{z}_{k} \| = 1. \]
Then for $0 < \omega < 2$, we have
\[ \lim_{k \to \infty} (I - \mathcal{P}^{\omega}_{\ell}) \vec{z}_{k} = 0. \]
\end{lemma}

\begin{lemma} \label{L:Natt-3.2a}
Suppose $\{ \vec{z}_{k} \}$ is a sequence in $\mathbb{C}^{d}$ such that 
\[ \| \vec{z}_{k} \| \leq 1 \text{ and } \lim_{k \to \infty} \| \mathcal{P}^{\omega} \vec{z}_{k} \| = 1, \]
then for $0 < \omega < 2$, we have
\[ \lim_{k \to \infty} (I - \mathcal{P}^{\omega}) \vec{z}_{k} = 0. \]
\end{lemma}

\begin{proof}
Note that 
\[ \left(I - \mathcal{P}^{\omega} \right) \vec{z}_{k} = \sum_{\ell \in \mathcal{L}} w(\ell,r) \left( I - \mathcal{P}_{\ell}^{\omega} \right) \vec{z}_{k},\]
so it is sufficient to show that the hypotheses of Lemma \ref{L:Natt-3.2} are satisfied.  Since $\| \mathcal{P}_{\ell}^{\omega} \vec{z}_{k} \| \leq 1$ and $\sum_{\ell} w(\ell,r) = 1$, we have
\begin{equation*}
1 = \lim_{k \to \infty} \| \mathcal{P}^{\omega} \vec{z}_{k} \| \leq \lim_{k \to \infty} \sum_{\ell \in \mathcal{L}} w(\ell,r) \| \mathcal{P}_{\ell}^{\omega} \vec{z}_{k} \| \leq 1.
\end{equation*}
Thus, we must have $\lim \| \mathcal{P}_{\ell}^{\omega} \vec{z}_{k} \| = 1$ for every $\ell \in \mathcal{L}$.
\end{proof}




\begin{lemma} \label{L:Natt-3.4}
For $0 < \omega < 2$, we have
\begin{equation}
\mathcal{N} (I - \mathcal{P}^{\omega}) = \bigcap_{v \text{ \emph{node} } } \mathcal{N}(I - P_{v}).
\end{equation}
\end{lemma}

\begin{proof}
Suppose $P_{v} \vec{z} = \vec{z}$ for every node $v$.  Then
\begin{equation*}
\mathcal{P}^{\omega} \vec{z} = \sum_{\ell \in \mathcal{L} } w(\ell, r) P_{\ell, p_{\ell}}^{\omega} \dots P_{\ell,1}^{\omega} \vec{z} = \sum_{\ell \in \mathcal{L} } w(\ell, r) \vec{z} = \vec{z}
\end{equation*}
thus the left containment follows.

Conversely, suppose that $\mathcal{P}^{\omega} \vec{z} = \vec{z}$.  Again, we obtain
\[ \| \vec{z} \| = \| \mathcal{P}^{\omega} \vec{z} \| \leq \sum_{\ell \in \mathcal{L} } w(\ell, r) \| P_{\ell, p_{\ell}}^{\omega} \cdots P_{\ell, 1}^{\omega} \vec{z} \| \leq \| \vec{z} \| \]
which implies that 
\[ P_{\ell, p_{\ell}}^{\omega} \cdots P_{\ell, 1}^{\omega} \vec{z} = \vec{z} \]
for every leaf $\ell$.  Hence, for every $\ell$, and every $j=1, \dots, p_{\ell}$, $P_{\ell, j}^{\omega} \vec{z} = \vec{z}$.
\end{proof}


\begin{lemma}[\cite{N-86}, Lemma V.3.5] \label{L:Natt-3.5}
For $0 < \omega < 2$, $( \mathcal{P}^{\omega} )^{k}$ converges strongly, as $k \to \infty$, to the orthogonal projection onto
\[ \bigcap_{v \text{ \emph{node} } } \mathcal{N}(I - P_{v} ) = \mathcal{N}(A). \]
\end{lemma}

The proof is identical to that in \cite{N-86}, using Lemmas \ref{L:Natt-V.3.1}, \ref{L:Natt-3.2a}, and \ref{L:Natt-3.4}.

\begin{proof}[Proof of Theorem \ref{Th:omega-c}]
Let $\vec{y}$ be any solution to the system of equations.  We claim that for any $\vec{z}$,
\begin{equation} \label{Eq:Q-omega} 
\mathcal{Q}^{\omega} \vec{z} = \mathcal{P}^{\omega} ( \vec{z} - \vec{y} ) + \vec{y}
\end{equation}
Indeed, for any nodes $v$ and $w$, and consequently for any leaf $\ell$, we have
\begin{align*} 
& \  & Q^{\omega}_{v} \vec{z} &= \vec{y} + P^{\omega}_{v} ( \vec{z} - \vec{y} )  \\
& \Rightarrow  & Q^{\omega}_{w} Q^{\omega}_{v} \vec{z} &= \vec{y} + P^{\omega}_{w} P^{\omega}_{v} ( \vec{z} - \vec{y} ) \\
& \Rightarrow  & \mathcal{Q}^{\omega}_{\ell} \vec{z} &= \vec{y} + \mathcal{P}^{\omega}_{\ell} ( \vec{z} - \vec{y} ) \\
& \Rightarrow  & \sum_{\ell \in \mathcal{L} } w(\ell, r) \mathcal{Q}^{\omega}_{\ell} \vec{z} &= \sum_{\ell \in \mathcal{L} } w(\ell, r) \left( \vec{y} + \mathcal{P}^{\omega}_{\ell}( \vec{z} - \vec{y} ) \right),
\end{align*}
which demonstrates Equation (\ref{Eq:Q-omega}).

Therefore, by Lemma \ref{L:Natt-3.5}, we have that for any $\vec{z}$,
\[ \left( \mathcal{Q}^{\omega} \right)^{k} \vec{z} \to \vec{y} + Pr ( \vec{z} - \vec{y} ), \]
as $k \to \infty$, where $Pr$ is the projection onto $\mathcal{N}(A)$.  If $\vec{z} \in \mathcal{R}( A^{*})$, we have that $\vec{y} + Pr ( \vec{z} - \vec{y} )$ is the unique solution to the system of equations that is in $\mathcal{R}(A^{*})$, and hence is the solution of minimal norm.
\end{proof}


We can see that for $\vec{z} \in \mathcal{R}(A^{*})$, the convergence rate of $\left( \mathcal{Q}^{\omega} \right)^{k} \vec{z} \to \vec{x}^{M}$ is linear, but we will formalize this in the next subsection (Corollary \ref{C:Fixed}).

\subsection{Inconsistent Equations}  \label{ssec:inconsistent}

We now consider the case of inconsistent systems of equations.  For this purpose, we must consider the relaxed version of the algorithm, as in the previous subsection.  Again, we assume $0<\omega<2$ and consider the limit
\[ \lim_{n \to \infty} \vec{x}^{(n)} = \vec{x}(\omega). \]
We will prove in Theorem \ref{Th:omega} and Corollary \ref{C:Fixed} that the limit exists, but unlike in the case of consistent systems, the limit will depend on $\omega$.  Moreover, we will prove in Theorem \ref{Th:omega-ic} that the limit
\[ \lim_{\omega \to 0} \vec{x}(\omega) = \vec{x}^{LS} \]
exists, and $\vec{x}^{LS}$ is a generalized solution which minimizes a weighted least-squares norm.  We follow the presentation of the analogous results for the classical Kaczmarz algorithm as presented in \cite{N-86}.  Indeed, we will proceed by analyzing the distributed  Kaczmarz algorithm using the ideas from Successive Over-Relaxation (SOR).  We need to follow the updates as they disperse through the tree, and also how the updates are pooled back at the root, and so we define the following quantities.  

We begin with reindexing the equations, which are currently indexed by the nodes as $S_{v} (\vec{x}) = b_{v}$.  As before, for each leaf $\ell$, we consider the path from the root $r$ to the leaf $\ell$, and index the corresponding equations as: 
\[ S_{\ell, 1}  (\vec{x}) = b_{\ell, 1}, \ \dots , S_{\ell, p_{\ell}} ( \vec{x} ) = b_{\ell, p_{\ell}}. \]

For each leaf $\ell$, we can define:
\begin{align*}
\mathcal{S}_{\ell} = \begin{pmatrix} S_{\ell, 1} \\ S_{\ell, 2} \\ \vdots \\ S_{\ell, p_{\ell}} \end{pmatrix}, \quad \vec{b}_{\ell} = \begin{pmatrix} b_{\ell, 1} \\ b_{\ell, 2} \\ \vdots \\ b_{\ell, p_{\ell}} \end{pmatrix}, \quad  \mathcal{D}_{\ell} = \begin{pmatrix} S_{\ell, 1} S_{\ell, 1}^{*} & 0 & \dots & 0 \\ 0 & S_{\ell, 2} S_{\ell, 2}^{*} & \dots & 0 \\ \vdots & \vdots & \ddots & \vdots \\ 0 &  0 & \dots & S_{\ell, p_{\ell}} S_{\ell, p_{\ell}}^{*} \end{pmatrix}
\end{align*}
and
\begin{align*}
\mathcal{L}_{\ell} = \begin{pmatrix} 0 & 0 & \dots & 0 & 0 \\ S_{\ell, 2} S_{\ell, 1}^{*} & 0 & \dots & 0 & 0 \\ \vdots & \vdots & \ddots & \vdots & \vdots \\ S_{\ell, p_{\ell}} S_{\ell, 1}^{*} &  S_{\ell, p_{\ell}} S_{\ell, 2}^{*} & \dots  & S_{\ell, p_{\ell}} S_{\ell, p_{\ell} - 1}^{*} & 0 \end{pmatrix}
\end{align*}
Then from input $\vec{x}^{(n)}$ at the root of the tree, the approximation at leaf $\ell$ after the dispersion stage in iteration $n$ is given by:
\begin{equation*}
\vec{x}^{(n)}_{\ell} = \mathcal{Q}^{\omega}_{\ell} \vec{x}^{(n)}  = \vec{x}^{(n)} + \sum_{j=1}^{p_{\ell}} S_{\ell, j}^{*} u_{j} = \vec{x}^{(n)} +  \mathcal{S}_{\ell}^{*} \vec{u},
\end{equation*}
where 
\begin{equation*}
\vec{u} := \left( u_{1} \dots u_{p_{\ell}} \right)^{T} = \omega \left( \mathcal{D}_{\ell} + \omega \mathcal{L}_{\ell} \right)^{-1} \left( \vec{b}_{\ell} - \mathcal{S}_{\ell} \vec{x}^{(n)} \right).
\end{equation*}
Therefore, we can write
\begin{align*}
 \vec{x}^{(n)}_{\ell} &= \vec{x}^{(n)} + \omega \mathcal{S}_{\ell}^{*} \left( \mathcal{D}_{\ell} + \omega \mathcal{L}_{\ell} \right)^{-1} \left( \vec{b}_{\ell} - \mathcal{S}_{\ell} \vec{x}^{(n)} \right) \\
 &= \left( I - \omega \mathcal{S}_{\ell}^{*} \left( \mathcal{D}_{\ell} + \omega \mathcal{L}_{\ell} \right)^{-1} \mathcal{S}_{\ell} \right) \vec{x}^{(n)} + \omega \mathcal{S}_{\ell}^{*} \left( \mathcal{D}_{\ell} + \omega \mathcal{L}_{\ell} \right)^{-1} \vec{b}_{\ell} .
\end{align*}

Combining these approximations back at the root yields:
\begin{align}
\vec{x}^{(n+1)} &= \sum_{\ell \in \mathcal{L}} w(\ell, r) \vec{x}^{(n)}_{\ell} \notag \\
&= \sum_{\ell \in \mathcal{L}} w(\ell, r) \left( I - \omega \mathcal{S}_{\ell}^{*} \left( \mathcal{D}_{\ell} + \omega \mathcal{L}_{\ell} \right)^{-1} \mathcal{S}_{\ell} \right) \vec{x}^{(n)} + \omega \sum_{\ell \in \mathcal{L}} w(\ell, r)  \mathcal{S}_{\ell}^{*} \left( \mathcal{D}_{\ell} + \omega \mathcal{L}_{\ell} \right)^{-1} \vec{b}_{\ell} \notag \\
&=  \left( I - \omega \sum_{\ell \in \mathcal{L}} w(\ell, r)  \mathcal{S}_{\ell}^{*} \left( \mathcal{D}_{\ell} + \omega \mathcal{L}_{\ell} \right)^{-1} \mathcal{S}_{\ell} \right) \vec{x}^{(n)} + \omega \sum_{\ell \in \mathcal{L}} w(\ell, r)  \mathcal{S}_{\ell}^{*} \left( \mathcal{D}_{\ell} + \omega \mathcal{L}_{\ell} \right)^{-1} \vec{b}_{\ell} \label{Eq:root-approx}.
\end{align}

We write
\begin{equation*}
\vec{x}^{(n+1)} = \sum_{\ell \in \mathcal{L}} w(\ell, r) \mathcal{B}^{\omega}_{\ell} \vec{x}^{(n)} + \sum_{\ell \in \mathcal{L} } w(\ell, r) \vec{\textbf{b}}^{\omega}_{\ell}
\end{equation*}
where
\begin{equation} \label{Eq:SOR-ell}
\mathcal{B}^{\omega}_{\ell} :=   I - \omega \mathcal{S}_{\ell}^{*} \left( \mathcal{D}_{\ell} + \omega \mathcal{L}_{\ell} \right)^{-1} \mathcal{S}_{\ell}; \qquad \vec{\textbf{b}}^{\omega}_{\ell} :=  \omega \mathcal{S}_{\ell}^{*} \left( \mathcal{D}_{\ell} + \omega \mathcal{L}_{\ell} \right)^{-1} \vec{b}_{\ell}.
\end{equation}
Written in this form, for each leaf $\ell$, the input at the root undergoes the linearly ordered Kaczmarz algorithm.  So, if the input at the root is $\vec{x}^{(n)}$, then the estimate at leaf $\ell$ is:
\[ \vec{x}^{(n)}_{\ell} = \mathcal{Q}^{\omega}_{\ell} \vec{x}^{(n)} = \mathcal{B}^{\omega}_{\ell} \vec{x}^{(n)} + \vec{\textbf{b}}^{\omega}_{\ell}. \]
As we shall see, for each leaf $\ell$ and $\omega \in (0,2)$, $\mathcal{B}^{\omega}_{\ell}$ has operator norm bounded by $1$, and the eigenvalues are either $1$ or strictly less than $1$ in magnitude.  We state these formally in Lemma \ref{L:B-Lip}.

We enumerate the leaves of the tree as $\ell_{1}, \dots, \ell_{t}$, and write: 
\begin{align*}
\mathcal{S} = \begin{pmatrix} \mathcal{S}_{\ell_{1}} \\ \vdots \\ \mathcal{S}_{\ell_{t}} \end{pmatrix}
\qquad 
\vec{\textbf{b}} = \begin{pmatrix} \vec{b}_{\ell_{1}} \\ \vdots \\ \vec{b}_{\ell_{t}} \end{pmatrix}
\end{align*}
The system of equations $A \vec{x} = \vec{b}$ becomes:
\begin{equation} \label{Eq:ampliated}
\mathcal{S} \vec{x} = \vec{\textbf{b}}
\end{equation}
where many of the equations are now repeated in Equation (\ref{Eq:ampliated}).  However, we have $\mathcal{N}(\mathcal{S}) = \mathcal{N}(A)$ and $\mathcal{R}(\mathcal{S}^{*}) = \mathcal{R}(A^{*})$.

We also write
\begin{equation} \label{Eq:mathcal-DL}
\mathcal{D} :=  \begin{pmatrix} \mathcal{D}_{\ell_{1}} & 0 & \dots & 0 \\ 0 & \mathcal{D}_{\ell_{2}} & \dots & 0 \\ \vdots & \vdots & \ddots & \vdots \\ 0  & 0 & \dots &  \mathcal{D}_{\ell_{t}} \end{pmatrix}
\qquad 
\mathcal{L} = \begin{pmatrix}  \mathcal{L}_{\ell_{1}} & 0 & \dots & 0 \\ 0 & \mathcal{L}_{\ell_{2}} & \dots & 0 \\ \vdots & \vdots & \ddots & \vdots \\ 0 & 0 & \dots &  \mathcal{L}_{\ell_{t}} \end{pmatrix}
\end{equation}
so
\begin{equation} \label{Eq:mathcal-DLinv}
\left( \mathcal{D}+ \omega \mathcal{L} \right)^{-1} = \begin{pmatrix} \left( \mathcal{D}_{\ell_{1}} + \omega \mathcal{L}_{\ell_{1}} \right)^{-1} & 0 & \dots & 0 \\ 0 & \left( \mathcal{D}_{\ell_{2}} + \omega \mathcal{L}_{\ell_{2}} \right)^{-1} & \dots & 0 \\ \vdots & \vdots & \ddots & \vdots \\ 0 & 0 & \dots & \left( \mathcal{D}_{\ell_{t}} + \omega \mathcal{L}_{\ell_{t}} \right)^{-1}
\end{pmatrix}
\end{equation}
We also define
\begin{equation} \label{Eq:mathcal-W}
\mathcal{W} =  \begin{pmatrix} w(\ell_{1}, r) I_{p_{\ell_{1}}}  & 0 & \dots & 0 \\ 0 & w(\ell_{2}, r) I_{p_{\ell_{2}}} & \dots & 0 \\ \vdots & \vdots & \ddots & \vdots \\ 0 & 0 & \dots & w(\ell_{t}, r) I_{p_{\ell_{t}}}
\end{pmatrix}
\end{equation}
Note that since $\mathcal{D} + \omega \mathcal{L}$ and $\mathcal{W}$ are block matrices with blocks of the same size, and in $\mathcal{W}$ the blocks are scalar multiples of the identity, we have that the two matrices commute:
\begin{equation} \label{Eq:commute}
\left( \mathcal{D}+ \omega \mathcal{L} \right)^{-1} \mathcal{W} = \mathcal{W} \left( \mathcal{D}+ \omega \mathcal{L} \right)^{-1} = \mathcal{W}^{1/2} \left( \mathcal{D}+ \omega \mathcal{L} \right)^{-1} \mathcal{W}^{1/2}.
\end{equation}

We can therefore write Equation (\ref{Eq:root-approx}) as
\begin{align*}
\vec{x}^{(n+1)} &\ = \left( I -  \omega \mathcal{S}^{*} \left( \mathcal{D} + \omega \mathcal{L} \right)^{-1} \mathcal{W} \mathcal{S} \right) \vec{x}^{(n)} + \omega \mathcal{S}^{*} \left( \mathcal{D} + \omega \mathcal{L} \right)^{-1} \mathcal{W} \vec{\textbf{b}}.  \\
&:= \mathcal{B}^{\omega} \vec{x}^{(n)} + \vec{\textbf{b}}^{\omega}.
\end{align*}

Note that $\mathcal{R}(\mathcal{S}^{*})$ is an invariant subspace for $\mathcal{B}^{\omega}$, and that $\vec{\textbf{b}}^{\omega} \in \mathcal{R}(\mathcal{S}^{*})$.  We let $\widehat{\mathcal{B}}^{\omega}$ denote the restriction of $\mathcal{B}^{\omega}$ to the subspace $\mathcal{R}(\mathcal{S}^{*})$.  As we shall see, provided the input $\vec{x}^{0} \in \mathcal{R}(\mathcal{S}^{*})$, the sequence $\vec{x}^{k}$ converges.  In fact, we will show that the transformation $\widehat{\mathcal{B}}^{\omega}$ is a contraction, and since $\vec{\textbf{b}}^{\omega} \in \mathcal{R}(\mathcal{S}^{*})$, then the mapping
\[ \vec{z} \mapsto \widehat{\mathcal{B}}^{\omega} \vec{z} + \vec{\textbf{b}}^{\omega} \]
has a unique fixed point within $\mathcal{R}(\mathcal{S}^{*})$.  We shall do so via a series of lemmas.

\begin{lemma} \label{L:B-Lip}
For each leaf $\ell$ and for $\omega \in (0,2)$, $\mathcal{B}^{\omega}_{\ell}$ is Lipschitz continuous with constant at most $1$ (i.e. it has operator norm at most $1$).  Consequently, $\widehat{\mathcal{B}}^{\omega}$ is also Lipschitz continuous with constant at most $1$.

Moreover, for each leaf $\ell$ and $\omega \in (0,2)$, if $\lambda$ is an eigenvalue of $\mathcal{B}^{\omega}_{\ell}$ with $| \lambda | = 1$, then $\lambda = 1$.  Consequently, any eigenvalue $\lambda \neq 1$ has the property $| \lambda | < 1$.
\end{lemma}

\begin{proof}
For input $\vec{z}_{i}$, we have that 
\[ \mathcal{Q}^{\omega}_{\ell} \vec{z}_{i} = \mathcal{B}^{\omega}_{\ell} \vec{z}_{i} + \vec{\textbf{b}}^{\omega}_{\ell}, \]
hence
\begin{equation*}
\| \mathcal{B}^{\omega}_{\ell} \vec{z}_{1} - \mathcal{B}^{\omega}_{\ell} \vec{z}_{2} \| = \| \mathcal{Q}^{\omega}_{\ell} \vec{z}_{1} - \mathcal{Q}^{\omega}_{\ell} \vec{z}_{2} \|  \leq \| \vec{z}_{1} - \vec{z}_{2} \|
\end{equation*}
by Lemma \ref{L:Q-Lip}.  Since $\mathcal{B}^{\omega}$ is a convex combination of the $\mathcal{B}^{\omega}_{\ell}$, it also has Lipschitz constant at most $1$.  The last conclusion follows from \cite[Lemma V.3.9]{N-86}.
\end{proof}

\begin{theorem} \label{Th:omega}
The spectral radius of $\widehat{\mathcal{B}}^{\omega}$ is strictly less than $1$.
\end{theorem}

\begin{proof}
For each leaf $\ell$, Lemma \ref{L:B-Lip} implies that
\begin{equation} \label{Eq:norm-1}
\| \mathcal{B}^{\omega}_{\ell} \| \leq 1, \qquad | \langle \mathcal{B}^{\omega}_{\ell} \vec{v}, \vec{v} \rangle | \leq \| \vec{v} \|^2. 
\end{equation}
Let $\lambda$ be an eigenvalue for $\widehat{\mathcal{B}}^{\omega}$.  We must have $\lambda \neq 1$; if it were not so, then there exists a nonzero $\vec{z} \in \mathcal{R}(\mathcal{S}^{*})$ with $\widehat{\mathcal{B}}^{\omega} \vec{z} = \vec{z}$.  However, by Lemma \ref{L:Natt-3.4} we must have $\vec{z} \in \mathcal{N}(A) = \mathcal{N}(S)$ which is a contradiction.  Let $\vec{v}$ be a unit norm eigenvector for $\lambda$.  We have
\begin{equation*}
| \lambda | = | \langle \widehat{\mathcal{B}}^{\omega} \vec{v}, \vec{v} \rangle | \leq \sum_{\ell \in \mathcal{L}} w(\ell, r) | \langle \mathcal{B}_{\ell}^{\omega} \vec{v}, \vec{v} \rangle | \leq 1.
\end{equation*}

Now suppose that $| \lambda | = 1$, then we similarly obtain
\begin{equation}
\lambda = \sum_{\ell \in \mathcal{L} } \langle \mathcal{B}_{\ell}^{\omega} \vec{v}, \vec{v} \rangle
\end{equation}
from which we deduce that the argument of the complex number $\langle \mathcal{B}_{\ell}^{\omega} \vec{v}, \vec{v} \rangle$ is independent of the leaf $\ell$.  Therefore, we must have for every leaf $\ell$
\begin{equation} \label{Eq:lambda-ell}
\langle \mathcal{B}_{\ell}^{\omega} \vec{v}, \vec{v} \rangle = \lambda.
\end{equation}
However, we know by the Cauchy-Schwarz inequality that equality in Equation (\ref{Eq:lambda-ell}) can only occur when $(\vec{v}, \lambda)$ is an eigenvector/eigenvalue pair for $\mathcal{B}_{\ell}^{\omega}$.  However,  Lemma \ref{L:B-Lip} implies that none of the leaves $\ell$ have the property that $\lambda$ is an eigenvalue, so we have arrived at a contradiction.
\end{proof}

\begin{corollary} \label{C:Fixed}
For $\omega \in (0,2)$ and for any initial input $\vec{x}^{(0)} \in \mathcal{R}(\mathcal{S}^{*})$, we have that the sequence given by
\begin{equation} \label{Eq:Bw-fixed}  
\vec{x}^{(n+1)} = \widehat{\mathcal{B}}^{\omega} \vec{x}^{(n)} + \vec{\textbf{\textup{b}}}^{\omega} 
\end{equation}
converges to a unique point in $\mathcal{R}(\mathcal{S}^{*})$, independent of $\vec{x}^{(0)}$, and the convergence rate is linear.
\end{corollary}

The following can be found in \cite[Theorem IV.1.1]{N-86}:

\begin{lemma} \label{L:omega-ic}
For each $\omega \in (0,2)$, let 
\[ \vec{x}(\omega) = \lim_{n \to \infty} \vec{x}^{(n)} \]
where $\vec{x}^{(n)}$ are as in Equation (\ref{Eq:Bw-fixed}).  Then, $\vec{x}(\omega)$ is the unique vector that satisfies the conditions
\begin{equation} \label{Eq:x-omega}
 \mathcal{S}^{*} \left(  \mathcal{D} + \omega \mathcal{L}  \right)^{-1} \mathcal{W} \left( \vec{\textbf{\textup{b}}} - \mathcal{S} \vec{z} \right) = 0; \qquad \vec{z} \in \mathcal{R}(\mathcal{S}^{*}).
\end{equation}
\end{lemma}

\begin{theorem} \label{Th:omega-ic}
For each $\omega \in (0,2)$, let
\[ \vec{x}(\omega) = \lim_{n \to \infty} \vec{x}^{n} \]
as in Equation (\ref{Eq:Bw-fixed}).  Then,
\[ \lim_{\omega \to 0} \vec{x}(\omega) = \vec{x}^{LS} \]
where $\vec{x}^{LS}$ minimizes the functional
\begin{equation} \label{Eq:functional} 
\vec{z} \mapsto \langle \mathcal{D}^{-1} \mathcal{W} ( \vec{\textbf{\textup{b}}} - \mathcal{S} \vec{z} ) , ( \vec{\textbf{\textup{b}}} - \mathcal{S} \vec{z} ) \rangle.
\end{equation}
\end{theorem}

\begin{proof}
Let $\vec{x}^{LS}$ be the unique vector that satisfies the conditions
\begin{equation} \label{Eq:x-LS} 
\mathcal{S}^{*} \mathcal{D}^{-1} \mathcal{W} \left( \vec{\textbf{b}} - \mathcal{S} \vec{x}^{LS} \right) = 0; \quad \vec{x}^{LS} \in \mathcal{R}(\mathcal{S}^{*}).
\end{equation}

We have that $\vec{x}(\omega)$, as the unique solution of Equation (\ref{Eq:x-omega}) and $\vec{x}^{LS}$, as the unique solution of Equation (\ref{Eq:x-LS}), satisfy
\[ \vec{x}(\omega) = \vec{x}^{LS} + 0(\omega). \]
Indeed, this follows from the fact that $\left( \mathcal{D} + \omega \mathcal{L} \right)^{-1} \to \mathcal{D}^{-1}$ as $\omega \to 0$, together with the fact that $\vec{x}(\omega), \vec{x}^{LS} \in \mathcal{R}(\mathcal{S}^{*})$.
\end{proof}

We can re-write Equation (\ref{Eq:functional}) in the following way:
\begin{equation} \label{Eq:new-functional}
\vec{z} \mapsto \langle D^{-1} V ( \vec{b} - A \vec{z} ) , ( \vec{b} - A \vec{z} ) \rangle
\end{equation}
where $D$ is the diagonal matrix with entries given by $\| \vec{a}_{v} \|^{2}$, and $V$ is the diagonal matrix whose entry for node $v$ is given by:
\[ V_{vv} = \sum_{\substack{ \ell \in \mathcal{L}  \\ \ell \prec v } } w(\ell, r). \]

\begin{remark}
As we mentioned in our discussion of related work, we can view the Kaczmarz algorithm that we have defined for tree-based data as a distributed optimization problem.  In this view, the objective function is given by Equation (\ref{Eq:new-functional}).  We emphasize here that, unlike existing distributed gradient descent algorithms~\cite{nedic2009distributed,yuan2016convergence}, we are able to establish convergence without the strong convexity assumption.  Indeed, in the case of real data, the Hessian of our objective function is $A^{*}D^{-1} V A$, which is nonnegative but need not be strictly positive.  Moreover, our convergence guarantee is valid in the complex case.
\end{remark}

\subsection{Distributed Solutions}

For each node $v$ in the tree, the sequence of approximations $\vec{x}_{v}^{(n)}$ and $\vec{y}_{v}^{(n)}$ will have a limit, i.e. the following limits exist:
\begin{equation} \label{Eq:vertex-limit}
\lim_{n \to \infty} \vec{x}_{v}^{(n)} = \vec{x}_{v}; \qquad \lim_{n \to \infty} \vec{y}_{v}^{(n)} = \vec{y}_{v}.
\end{equation}
In the relaxed case, these limits may depend on the relaxation parameter $\omega$; if so we will denote this dependence by $\vec{x}_{v}(\omega)$ and $\vec{y}_{v}(\omega)$.

\begin{corollary}
If the system of equations $A \vec{x} = \vec{b}$ is consistent, then for every node $v$ and every $\omega \in (0,2)$, the limits $\vec{x}_{v}$ and $\vec{y}_{v}$ as in Equation (\ref{Eq:vertex-limit}) equal $\vec{x}^{M}$, the solution of minimal norm.
\end{corollary}

\begin{proof}
We have by Theorem \ref{Th:omega-c} that $\vec{x}(\omega) = \vec{x}^{M}$ for every $\omega \in (0,2)$.  For a node $v$, let the path from the root $r$ to $v$ be denoted by $r = (v, 1), \dots, (v, p_{v}) = v$, where $p_{v} - 1$ is the length of the path.  Then, we have that
\begin{equation*}
 \lim_{n \to \infty} \vec{x}_{v}^{(n)} = \lim_{n \to \infty} Q_{v,p_{v}}^{\omega} \cdots Q_{v,1}^{\omega} \vec{x}^{(n)} = Q_{v,p_{v}}^{\omega} \cdots Q_{v,1}^{\omega} \vec{x}(\omega) = \vec{x}^{M}.
\end{equation*}
This holds as a consequence of the fact that any solution to the system of equations is fixed by $Q_{(\cdot)}^{\omega}$. 

Since we have that $\vec{y}_{v}^{(n)}$ is a convex combination of the vectors $\vec{x}_{\ell}^{(n)}$, which all converge to $\vec{x}(\omega)$, we have that $\vec{y}_{v} = \vec{x}^{M}$ also.
\end{proof}

\begin{corollary}
If the system of equations $A \vec{x} = \vec{b}$ is inconsistent, then for every node $v$ and every $\omega \in (0,2)$, the limits $\vec{x}_{v}$ and $\vec{y}_{v}$ as in Equation (\ref{Eq:vertex-limit}) exist and depend on $\omega$.  Moreover, we have
\begin{equation*}
\lim_{\omega \to 0} \vec{x}_{v}(\omega) = \vec{x}^{LS} \qquad \lim_{\omega \to 0} \vec{y}_{v}(\omega) = \vec{x}^{LS},
\end{equation*}
where $\vec{x}^{LS}$ is the vector as in Theorem \ref{Th:omega-ic}.
\end{corollary}

\begin{proof}
We apply the SOR analysis of $\vec{x}_{v}^{(n)} = Q_{(v,p_{v})}^{\omega} \cdots Q_{(v,1)}^{\omega} \vec{x}^{(n)}$ with input $\vec{x}^{(n)}$ to obtain
\begin{equation*}
\vec{x}_{v}^{(n)} = \mathcal{B}_{v}^{\omega} \vec{x}^{(n)} + \vec{\textbf{b}}_{v}^{\omega}
\end{equation*}
where $\mathcal{B}_{v}^{\omega}$ and $\vec{\textbf{b}}_{v}^{\omega}$ are analogous to those in Equation (\ref{Eq:SOR-ell}).  Taking limits on $n$, we obtain
\begin{equation*}
\vec{x}_{v}(\omega) = \mathcal{B}_{v}^{\omega} \vec{x}(\omega) + \vec{\textbf{b}}_{v}^{\omega}.
\end{equation*}
Since, as $\omega \to 0$, we have that $\mathcal{B}_{v}^{\omega} \to I$, $\vec{\textbf{b}}_{v}^{\omega} \to 0$, and $\vec{x}(\omega) \to \vec{x}^{LS}$ , we obtain
\begin{equation*}
\lim_{\omega \to 0} \vec{x}_{v}(\omega) = \vec{x}^{LS}.
\end{equation*}

As previously, $\vec{y}_{v}(\omega)$ is a convex combination of $\vec{x}_{\ell}(\omega)$, so $\vec{y}_{v}(\omega) \to \vec{x}^{LS}$ as $\omega \to 0$ also.
\end{proof}



\subsection{Error Analysis}

We consider the question of how errors propagate through the iterations of the dispersion and pooling stages.  We model errors as additive; the sources of errors could be machine errors, transmission errors, errors from compression to reduce communication complexity, etc.  Additive errors then take on the form
\begin{equation} \label{Eq:additive}
\vec{x}_{v,e}^{(n)} = \vec{x}_{v}^{(n)} + \epsilon_{v}^{(n)}; \qquad \vec{y}_{v,e}^{(n)} = \vec{y}_{v}^{(n)} + \delta_{v}^{(n)}
\end{equation}
Here, $\vec{x}_{v,e}^{(n)}$ and $\vec{y}_{v,e}^{(n)}$ are the error-riddled estimates which are passed to the successor (or predecessor) nodes in the dispersion (or pooling) stage, respectively, with additive errors $\epsilon_{v}^{(n)}$ and $\delta_{v}^{(n)}$.  We trace the errors during the dispersion stage as follows:  for vertex $v$ on a path between the root $r$ and leaf $\ell$, and the path parameterized by $r = (\ell, 1) , \dots, (\ell, p_{\ell}) = \ell$, suppose that $v = (\ell, k)$.  Then, the error introduced at vertex $v$ (with errors introduced at no other vertex) results in the estimate
\begin{align}
\vec{x}_{\ell,e}^{(n)} &= Q_{\ell,p_{\ell}}^{\omega} \cdots Q_{\ell, k+1}^{\omega}( \vec{x}_{v}^{(n)} + \epsilon_{v}^{(n)})  \notag \\
&= Q_{\ell,p_{\ell}}^{\omega} \cdots Q_{\ell, k+1}^{\omega}( \vec{x}_{v}^{(n)} ) + \widetilde{\epsilon}_{v,\ell}^{(n)} \label{Eq:error-1} \\ 
&= \mathcal{Q}_{\ell}^{\omega}( \vec{x}^{(n)} ) + \widetilde{\epsilon}_{v,\ell}^{(n)}. \notag
\end{align}
Equation (\ref{Eq:error-1}) follows for some $\widetilde{e}_{v}^{(n)}$ since the $Q_{(\cdot)}^{\omega}$ are affine transformations.  We have that 
\begin{equation*}
\|  \widetilde{\epsilon}_{v}^{(n)} \| = \| Q_{\ell,p_{\ell}}^{\omega} \cdots Q_{\ell, k+1}^{\omega}( \vec{x}_{v}^{(n)} + \epsilon_{v}^{(n)}) - Q_{\ell,p_{\ell}}^{\omega} \cdots Q_{\ell, k+1}^{\omega}( \vec{x}_{v}^{(n)} ) \| \leq \| \epsilon_{v}^{(n)} \|
\end{equation*}
since the $Q_{(\cdot)}^{\omega}$ have Lipschitz constant $1$.  The additive errors $\delta_{v}^{(n)}$ simply sum in the pooling stage, and thus we calculate the total errors from iteration $n$ to iteration $n+1$.
\begin{lemma}
Suppose we have additive errors as in Equation (\ref{Eq:additive}) introduced in iteration $n$.  Suppose no errors were introduced in previous iterations.  Then the estimate after iteration $n$ is:
\begin{equation}
\vec{x}_{e}^{(n+1)} = \vec{x}^{(n+1)} + \sum_{v} \sum_{\substack{ \ell \in \mathcal{L}  \\ \ell \prec v } } w(\ell, r) \tilde{e}_{v}^{(n)} + \sum_{v} w(v,r) \delta_{v,e}^{(n)}.
\end{equation}
The magnitude of the error is bounded by:
\begin{equation}
\| \vec{x}_{e}^{(n+1)} - \vec{x}^{(n+1)} \| \leq K \text{ max }\{ \| \epsilon_{v}^{(n)} \|, \| \delta_{v}^{(n)} \| \} 
\end{equation}
where $K$ is $2$ times the depth of the tree.
\end{lemma}

We write
\begin{equation} \label{Eq:total-error}
E^{(n)} = \sum_{v} \sum_{\substack{ \ell \in \mathcal{L}  \\ \ell \prec v } } w(\ell, r) \tilde{e}_{v}^{(n)} + \sum_{v} w(v,r) \delta_{v,e}^{(n)}.
\end{equation}

\begin{theorem}
If the additive errors in Equation (\ref{Eq:additive}) are uniformly bounded by $M$, and the system of equations $A \vec{x} = \vec{b}$ has a unique solution.  
Then the sequence of approximations $\{\vec{x}_{e}^{(n)} \}$ has the property that
\begin{equation} \label{Eq:error-est}
\limsup_{n \to \infty} \| \vec{x}(\omega) - \vec{x}_{e}^{(n)} \| \leq  \dfrac{2 K M}{1 - \rho (\mathcal{B}^{\omega} ) }
\end{equation}
where $K$ is the depth of the tree.
\end{theorem}

\begin{proof}
We have
\begin{equation*}
\vec{x}_{e}^{(n)}  = \vec{x}^{(n)} + \sum_{k=1}^{n} \left( \mathcal{B}^{\omega} \right)^{n-k} E^{(k)}.
\end{equation*}
As noted previously, $\| E^{(k)} \| \leq 2 K M$, and if $A \vec{x} = \vec{b}$ has a unique solution, then $\rho (\mathcal{B}^{\omega} ) < 1$ (see proof of Theorem \ref{Th:unique-soln}).

Thus, for any matrix norm $\| \cdot \|$ with $\rho (\mathcal{B}^{\omega} ) < \| \mathcal{B}^{\omega} \|$
\begin{equation*}
\| \vec{x}^{(n)} - \vec{x}_{e}^{(n)} \| \leq \sum_{k=0}^{n-1} 2 K M \| \mathcal{B}^{\omega} \|^{k}
\end{equation*}
from which Equation (\ref{Eq:error-est}) follows.
\end{proof}

If the system of equations does not have a unique solution, then the mapping $ \mathcal{B}^{\omega} $ has $1$ as an eigenvalue, and so the parts of the errors that lie in that eigenspace accumulate.  Hence, no stability result is possible in this case.

\subsection{Extensions}

We present several possible extensions and variations that require only minor modifications to the proofs of Theorems \ref{Th:unique-soln}, \ref{Th:omega-c}, and \ref{Th:omega-ic}.

The first variation is when the nodes of the tree contain more than one equation from $A \vec{x} = \vec{b}$.  This can be easily modeled under the assumption that each node proceeds through its equations in some \emph{a priori} fixed linear ordering, and subsequently in the tree replacing each node with a path.  Again, the SOR analysis passes through unaltered.  Alternatives to fixed linear orderings in this situation will not be considered here.

The second variation is when the data for each node consists of linear transformations $T_{v} : H \to H_{v}$ rather than linear functionals $S_{v}: H \to \mathbb{C}$.  If we assume that at each node, $T_{v}$ is onto \cite{N-86}, then again the SOR analysis passes through unaltered, and so we will not consider this variation further here.

The third variation is to perform the Kaczmarz update during the pooling stage of the iteration.   This variation, however, requires more than minor modifications to the proofs, and will thus be considered elsewhere.

\section{Implementation and Examples} \label{S:examples}

For the standard Kaczmarz algorithm, it is well known that the method converges if and only if the relaxation parameter $\omega$ is in the interval $(0,2)$.  For our distributed Kaczmarz, the situation is not nearly as clear.  The proofs of Theorems \ref{Th:omega-c} and \ref{Th:omega-ic} require that $\omega \in (0,2)$, but in numerical experiments, convergence occurred for $\omega \in (0,\Omega)$ for some $\Omega \ge 2$. The largest $\Omega$ observed was around 3.8. The precise upper limit depends on the equations themselves.  In this section, we perform a preliminary analysis of the computation of $\Omega$ and the optimal $\omega_{opt}$ for a very simple setup, and give numerical results for several examples.

\subsection{Examples}
\begin{example}
We consider the matrix
\begin{equation*}
  A =
  \begin{pmatrix}
    -\sin \alpha & \cos \alpha \\
    0 & 1
  \end{pmatrix}.
\end{equation*}
In geometric terms, the Kaczmarz method for this example corresponds
to projection onto the $x$-axis and onto a line forming an angle
$\alpha$ with the $x$-axis.

For standard Kaczmarz, the iteration matrix is
\begin{equation*}
    B_\omega = I - \omega A^* (D+\omega L)^{-1} A =
    \begin{pmatrix}
      1 - \omega \sin^2 \alpha & \omega \sin \alpha \cos \alpha \\
      \omega (1 - \omega) \sin \alpha \cos \alpha & (1-\omega) (1 -
      \omega \cos^2 \alpha)
    \end{pmatrix}.
\end{equation*}
The eigenvalues are
\begin{equation*}
  \lambda = \left[ \frac{\omega^2 \cos^2 \alpha}{2} + (1 - \omega)
    \right] \pm \omega \cos\alpha \sqrt{(\omega-2)^2 - \omega^2 \sin^2 \alpha}.
\end{equation*}
For small $\omega$, the eigenvalues are real and decreasing as a
function of $\omega$. They become complex at
\begin{equation*}
  \omega_{opt} = \frac{2}{1+\sin \alpha},
\end{equation*}
which is between 1 and 2. After that point, both eigenvalues have
magnitude $\omega - 1$, and the spectral radius increases in a
straight line.  The dependence of $\rho$ on $\omega$ is illustrated
below in the left half of fig.~\ref{figure:example2}. Here $\alpha =
\pi/3$, $\omega_{opt} \approx 1.0718$, $\rho_{opt} \approx 0.0718$.

As pointed out in~\cite{N-86}, there is a strong connection between the
classical Kaczmarz method and Successive Over-Relaxation (SOR). In SOR
the relationship between $\omega$ and $\rho$ shows the same type of
behavior.

The example with two equations is too small to implement as
distributed Kaczmarz, but we consider something similar. We project
the same $\vec{x}^{(n)}$ onto each line, and average the result to get
$\vec{x}^{(n+1)}$. We will refer to this as the {\em averaged Kaczmarz
  method}.

The iteration matrix is
\begin{equation*}
  B_\omega =
  \begin{pmatrix}
    1 - \frac{\omega}{2} \sin^2 \alpha & \frac{\omega}{2} \sin \alpha
    \cos \alpha \\
    \frac{\omega}{2} \sin \alpha \cos \alpha & \frac{\omega}{2} \sin^2
    \alpha - \omega + 1.
  \end{pmatrix}
\end{equation*}
The eigenvalues here are always real and vary linearly with $\omega$, namely
\begin{equation*}
  \lambda_{1,2} = 1 + \frac{\omega}{2} \left( \pm \cos \alpha - 1 \right).
\end{equation*}
They both have the value 1 at $\omega = 0$, and are both
decreasing with increasing $\omega$. The first one reaches $(-1)$ at
\begin{equation*}
  \Omega = \frac{4}{1 + \cos \alpha}.
\end{equation*}
Thus, the upper limit $\Omega$ is somewhere between 2 and 4, depending
on $\alpha$. In numerical experiments with the distributed Kaczmarz
method for larger matrices, we have observed $\Omega$ near 4, but
never above 4. We conjecture that $\Omega$ can never be larger than 4.

The minimum spectral radius occurs at $\omega_{opt} = 2$, independent
of $\alpha$, with $\rho_{opt} = \cos\alpha$. The dependence of $\rho$
on $\omega$ is illustrated below in the left half of
fig.~\ref{figure:example2}. In this example, the graph for the
averaged Kaczmarz method consists of two line segments, with
$\omega_{opt} = 2$, $\rho_{opt} = 0.5$.

Figure~\ref{figure:example1} illustrates the optimal $\omega$ for
$\alpha = \pi/2$.  The optimal $\omega$ for standard
Kaczmarz is $\omega = 1$, with $\rho = 0$. Convergence occurs in a
single step. For the averaged method, the optimal $\omega$ is 2, where
again convergence occurs in a single step. The averaged method would
still converge for a range of $\omega > 2$.

\begin{figure}[h]
  \centering
  \includegraphics[height=2.5in]{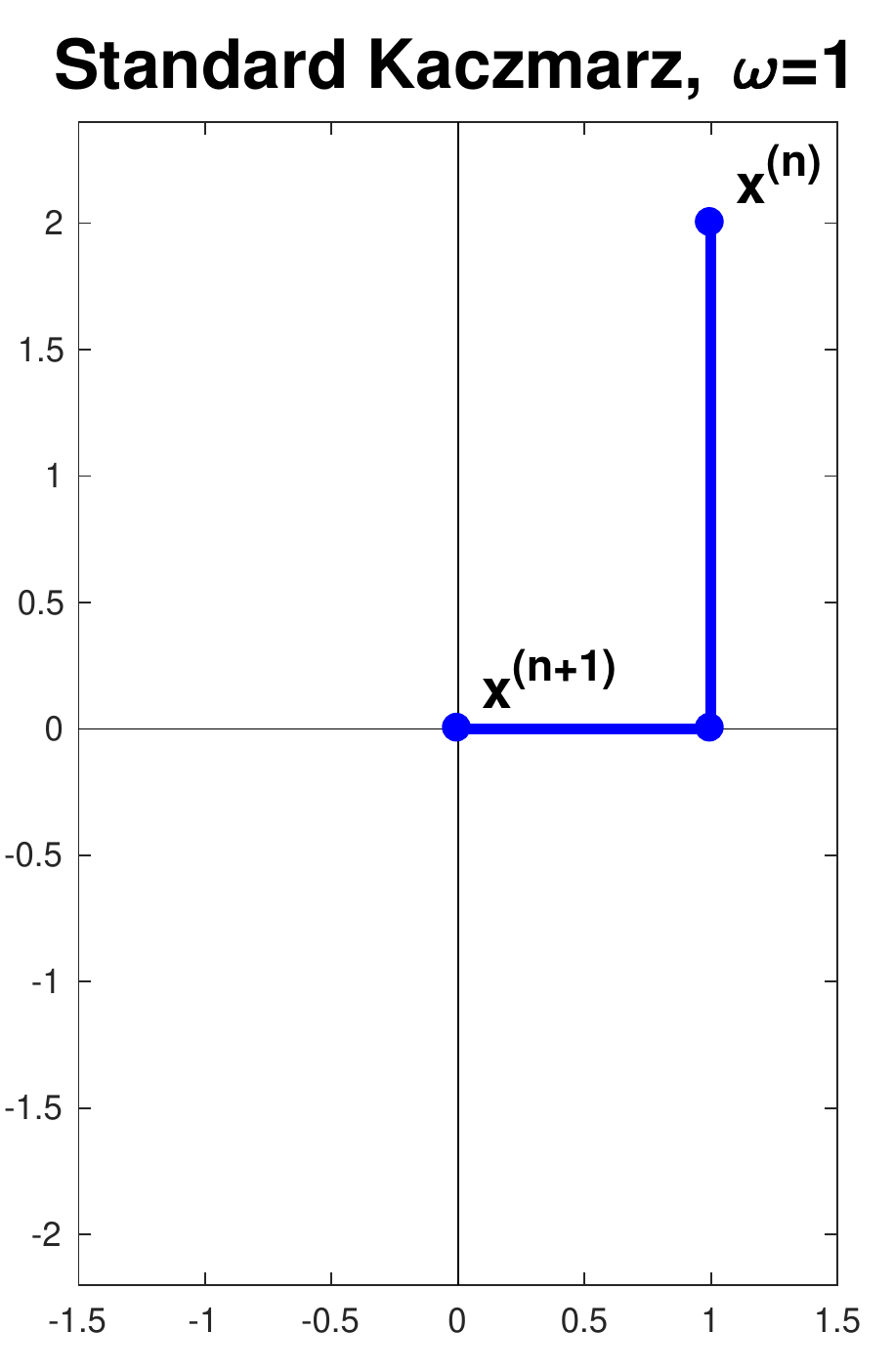} \quad
  \includegraphics[height=2.5in]{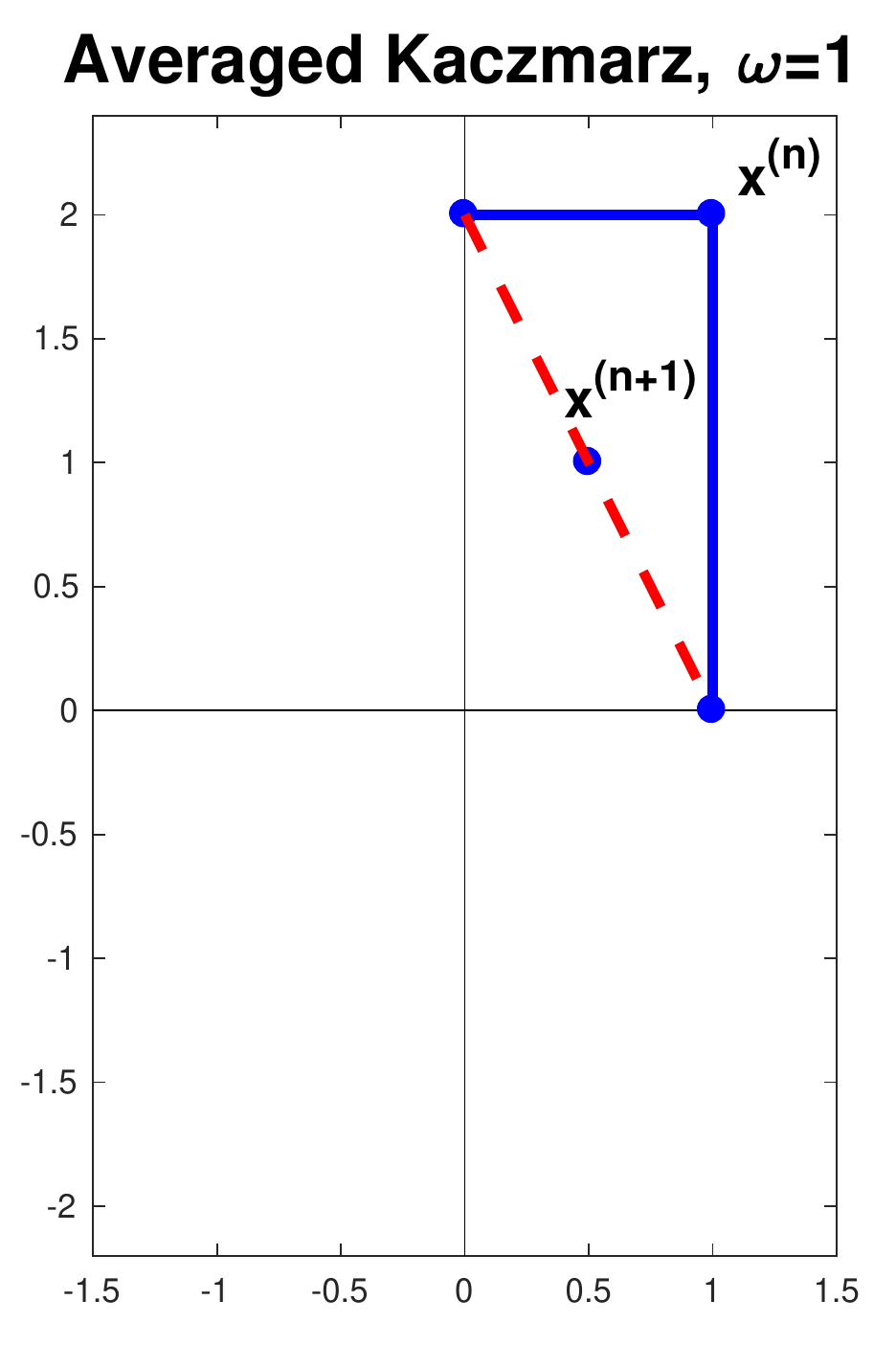} \quad
  \includegraphics[height=2.5in]{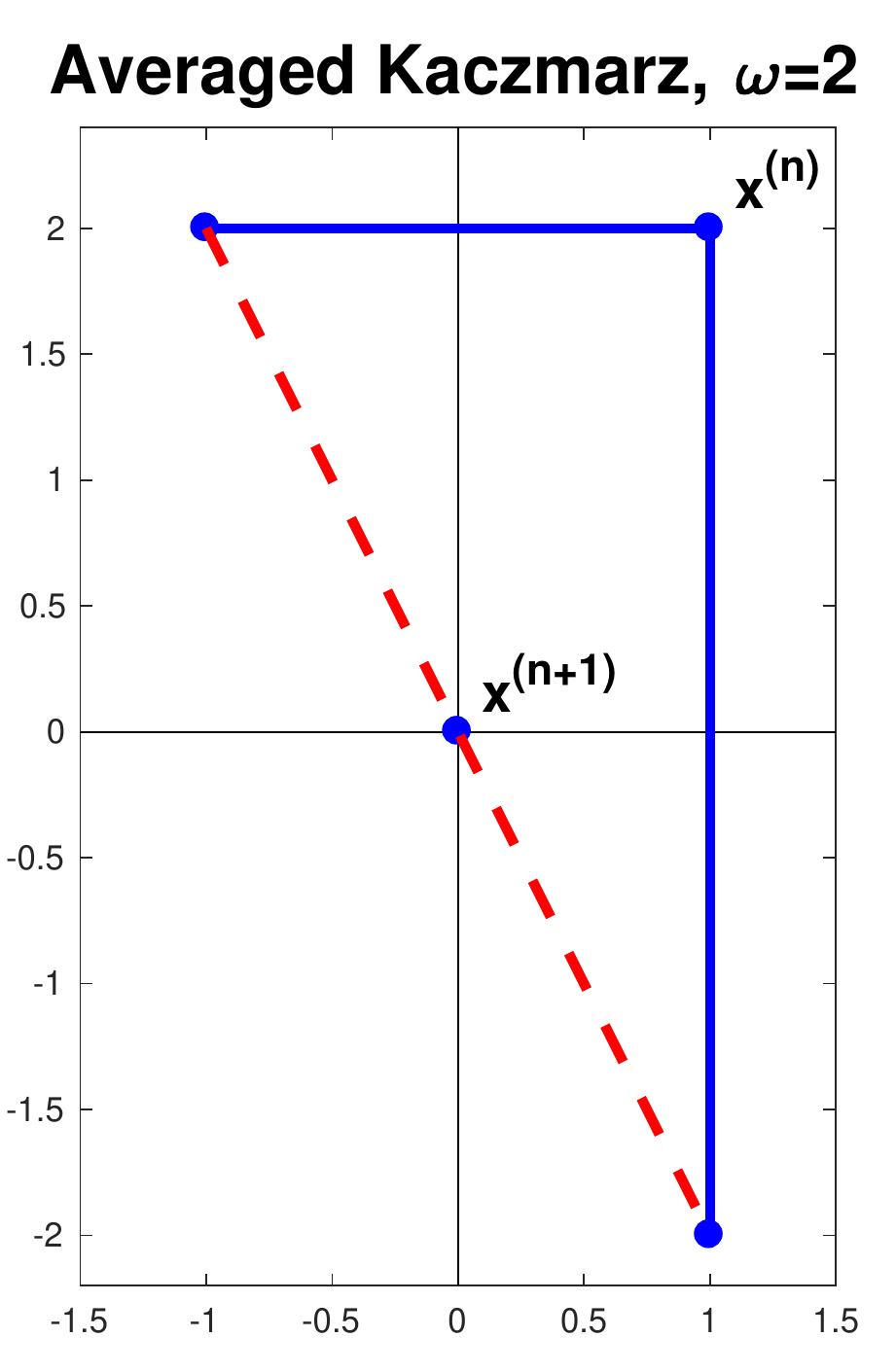}
  \caption{Example 1 with $\alpha = \pi/2$. The pictures show one step
    of standard Kaczmarz with $\omega=1$, and one step of averaged Kaczmarz
    for $\omega=1$ and $\omega=2$. This illustrates the need for a
    larger $\omega$ in the averaged Kaczmarz method.}
  \label{figure:example1}
\end{figure}  

Numerical experiments with larger sets of equations indicate that the
optimal $\omega$ for classical Kaczmarz is usually larger than 1, but
of course cannot exceed 2. The optimal $\omega$ for distributed
Kaczmarz is usually larger than 2, sometimes even approaching 4.
 \end{example}
  
\begin{example} \label{subsec:example2}

We used a random matrix of size $8 \times 8$, with entries generated
using a standard normal distribution. For the distributed Kaczmarz
method, we used the 8-node graph as shown on the right in Figure \ref{fig:graphs}.

For the standard Kaczmarz method, the optimal relaxation parameter was
$\omega_{opt} \approx 1.7354$, with spectral radius $\rho_{opt}
\approx 0.93147$. For the distributed Kaczmarz method, the results
were $\omega_{opt} \approx 3.7888$, with spectral radius $\rho_{opt}
\approx 0.99087$. This is illustrated in on the right in figure~\ref{figure:example2}.

\begin{figure}[h]
  \centering
  \includegraphics[width=2.5in]{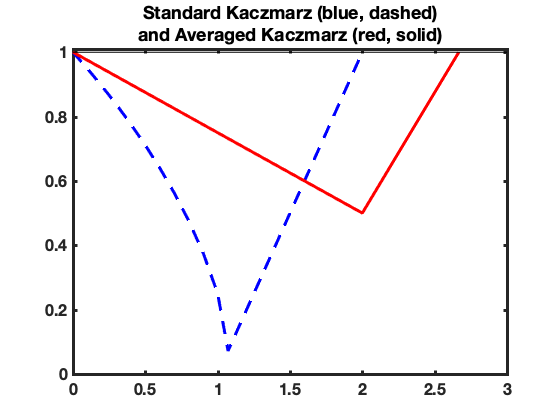} \quad
  \includegraphics[width=2.5in]{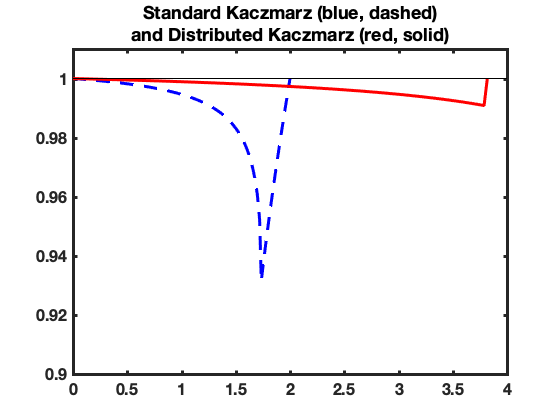}
  \caption{Dependence of the spectral radius $\rho$ of the iteration
    matrix on the relaxation parameter $\omega$. The left graph shows
    example 1 with $\alpha = \pi/3$. The right graph shows example 2.}
  \label{figure:example2}
\end{figure}  
\end{example}

\subsection{Implementation}

The implementation of the distributed Kaczmarz algorithm is based on
the Matlab Graph Theory toolbox. This toolbox provides support for
standard graphs and directed graphs (digraphs), weighted or
unweighted. We are using a weighted digraph.
The graph is defined by specifying the edges, which automatically also
defines the nodes. Specifying nodes is only necessary if there are
additional isolated nodes.
Both nodes and edges can have additional properties attached to them.
We take advantage of that by storing the equations and right-hand
sides, as well as the current approximate solution, in the nodes. The
weights are stored in the edges.
We are currently only considering tree-structured graphs. One node is
the root.  Each node other than the root has one incoming edge, coming
from the predecessor, and zero or more outgoing edges leading to the
successors. A node without a successor is called a {\em leaf}.

The basic Kaczmarz step has the form {\tt x\_new =
  update\_node(node,omega,x)}. The graph itself is a global data
structure, accessible to all subroutines; it would be very inefficient
to pass it as an argument every time.

The {\tt update\_node} routine does the following:

\begin{itemize}
    \item Use the equation(s) in the node to update {\tt x}
    \item Execute the {\tt update\_node} routine for each successor node
    \item Combine the results into a new {\tt x}, using the weights
      stored in the outgoing edges
    \item Return {\tt x\_new}
\end{itemize}

This routine needs to be called only once per iteration, for the
root. It will traverse the entire tree recursively.

\subsection{Numerical Experiments}
\label{sec:numerical}

We illustrate the methods with some simple numerical experiments. All
experiments were run with three different nonsingular matrices each,
of sizes $3 \times 3$ and $8 \times 8$.  All matrices were randomly generated once, and
then stored.  The right-hand size vectors are also random, and scaled
so that the true solution has $L^2$-norm 1. The test matrices are
\begin{itemize}
  \item An almost orthogonal matrix, generated from a random
    orthogonal matrix by truncating to one decimal of accuracy
  \item A random matrix, based on a standard normal distribution
  \item A random matrix, based on a uniform distribution in $[-1,1]$
\end{itemize}
In each case, we used the optimal $\omega$, based on minimizing the
spectral radius of the iteration matrix numerically. 
The distributed Kaczmarz method used the graphs shown in Figure \ref{fig:graphs}.  
Results are shown in Tables \ref{table:results3} and \ref{table:results8}.
In all cases, we start with $\boldx_0 = \boldzero$, so the initial
$L^2$-error is $e_0 = 1$. $e_{10}$ refers to the error after 10
iteration steps.
For an orthogonal matrix, the standard Kaczmarz method converges in a
single step. It is not surprising that it performs extremely well for
the almost orthogonal matrices.

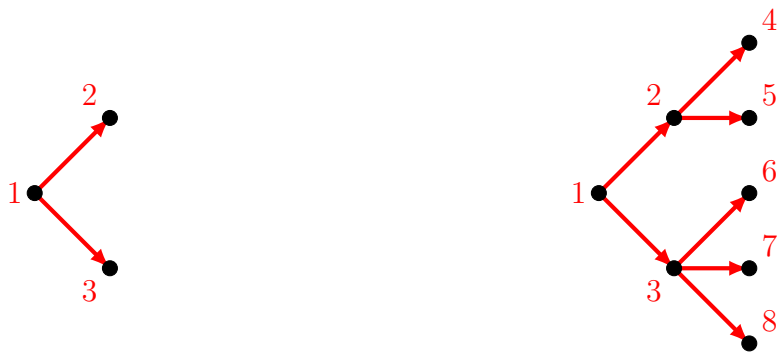
\begin{figure}[h]
  \centering
  \begin{tikzpicture}
    \coordinate (Origin)   at (0,0);

   \draw [ultra thick,-latex,red] (0,0) node [left] {$1$}
        -- (1,1) node [above left]  {$2$ };
   \draw [ultra thick,-latex,red] (0,0)
        -- (1,-1) node [below left] {$3$};

   \node[draw,circle,inner sep=2pt,fill] at (0,0) {};
   \node[draw,circle,inner sep=2pt,fill] at (1,1) {};
   \node[draw,circle,inner sep=2pt,fill] at (1,-1) {};



\begin{scope}[xshift=7.5cm,yshift=0.0cm]
    \coordinate (Origin)   at (0,0);

   \draw [ultra thick,-latex,red] (0,0) node [left] {1}
        -- (1,1) node [above left] {2};
   \draw [ultra thick,-latex,red] (0,0)
        -- (1,-1) node [below left] {3};
   \draw [ultra thick,-latex,red] (1,1) 
        -- (2,2) node [above right] {4};
   \draw [ultra thick,-latex,red] (1,1)
        -- (2,1) node [above right] {5};
   \draw [ultra thick,-latex,red] (1,-1)
        -- (2,0) node [above right] {6};
   \draw [ultra thick,-latex,red] (1,-1)
        -- (2,-1) node [above right] {7};
   \draw [ultra thick,-latex,red] (1,-1) 
        -- (2,-2) node [above right] {8};


   \node[draw,circle,inner sep=2pt,fill] at (0,0) {};
   \node[draw,circle,inner sep=2pt,fill] at (1,1) {};
   \node[draw,circle,inner sep=2pt,fill] at (1,-1) {};
   \node[draw,circle,inner sep=2pt,fill] at (2,2) {};
   \node[draw,circle,inner sep=2pt,fill] at (2,1) {};
   \node[draw,circle,inner sep=2pt,fill] at (2,-1) {};
   \node[draw,circle,inner sep=2pt,fill] at (2,0) {};
   \node[draw,circle,inner sep=2pt,fill] at (2,-2) {};

\end{scope}



 \end{tikzpicture}
  \caption{The two graphs used in numerical experiments with the distributed Kaczmarz method.}
  \label{fig:graphs}
\end{figure}


\begin{table}[h]
  \begin{center}
    \begin{tabular}{l|ccc|ccc|}
      & \multicolumn{3}{c|}{Standard Kaczmarz} &
      \multicolumn{3}{c|}{Distributed Kaczmarz} \\
      & $\omega_{opt}$ & $\rho_{opt}$ & $e_{10}$ & $\omega_{opt}$ & $\rho_{opt}$ & $e_{10}$ \\
      \hline
      orthogonal & 1.00030 & 0.00294 &                      0 & 1.33833 & 0.33753 & $1.5974 \cdot 10^{-5}$ \\
      normal     & 1.07213 & 0.20188 & $1.2793 \cdot 10^{-6}$ & 1.82299 & 0.29611 & $7.2461 \cdot 10^{-6}$ \\
      uniform    & 1.18634 & 0.37073 & $9.0922 \cdot 10^{-4}$ & 1.92714 & 0.82562 & $1.49608 \cdot 10^{-1}$ \\
      \hline
    \end{tabular}
  \end{center} 
  \caption{Numerical results for a $3 \times 3$ system of equations.}
  \label{table:results3}
\end{table}

\begin{table}[h]
  \begin{center}
    \begin{tabular}{l|ccc|ccc|}
      & \multicolumn{3}{c|}{Standard Kaczmarz} &
      \multicolumn{3}{c|}{Distributed Kaczmarz} \\
      & $\omega_{opt}$ & $\rho_{opt}$ & $e_{10}$ & $\omega_{opt}$ & $\rho_{opt}$ & $e_{10}$ \\
      \hline
      orthogonal & 1.01585 & 0.04931 & $1.53  \cdot 10^{-13}$ & 1.76733 & 0.73919 & $2.6757 \cdot 10^{-2}$ \\
      normal     & 1.73543 & 0.93147 & $8.5663 \cdot 10^{-1}$ & 3.78883 & 0.99087 & $9.0960 \cdot 10^{-1}$ \\
      uniform    & 1.88188 & 0.92070 & $7.1463 \cdot 10^{-1}$ & 3.73491 & 0.99890 & $7.7508 \cdot 10^{-1}$ \\
      \hline
    \end{tabular}
  \end{center}
  \caption{Numerical results for an $8 \times 8$ system of equations.}
  \label{table:results8}
\end{table}

\bigskip

\noindent \textbf{Acknowledgements.}  This research was supported by the National Science Foundation and the National Geospatial-Intelligence Agency under awards DMS-1830254 and CCF-1750920.


\begin{thebibliography}{10}

\bibitem{BT-97}
Dimitri~P. Bertsekas and John~N. Tsitsiklis, \emph{Parallel and distributed
  computation: Numerical methods}, Athena Scientific, Nashua, NH, 1997,
  originally published in 1989 by Prentice-Hall; available for free download at
  {\tt http://hdl.handle.net/1721.1/3719}.

\bibitem{boyd2011distributed}
Stephen Boyd, Neal Parikh, Eric Chu, Borja Peleato, Jonathan Eckstein, et~al.,
  \emph{Distributed optimization and statistical learning via the alternating
  direction method of multipliers}, Foundations and Trends{\textregistered} in
  Machine Learning \textbf{3} (2011), no.~1, 1--122.

\bibitem{chi2016kaczmarz}
Yuejie Chi and Yue~M Lu, \emph{Kaczmarz method for solving quadratic
  equations}, IEEE Signal Processing Letters \textbf{23} (2016), no.~9,
  1183--1187.

\bibitem{EHL-81}
P.~P.~B. Eggermont, G.~T. Herman, and A.~Lent, \emph{Iterative algorithms for
  large partitioned linear systems, with applications to image reconstruction},
  Linear Alg. Appl. \textbf{40} (1981), 37--67.

\bibitem{GBH-70}
Richard Gordon, Robert Bender, and Gabor Herman, \emph{Algebraic reconstruction
  techniques ({ART}) for threedimensional electron microscopy and x-ray
  photography}, Journal of Theoretical Biology \textbf{29} (1970), no.~3,
  471--481.

\bibitem{HS-78}
C.~Hamaker and D.~C. Solmon, \emph{The angles between the null spaces of {X}
  rays}, Journal of Mathematical Analysis and Applications \textbf{62} (1978),
  no.~1, 1--23.

\bibitem{Hansen2010}
Per~Christian Hansen, \emph{Discrete inverse problems}, Fundamentals of
  Algorithms, vol.~7, Society for Industrial and Applied Mathematics (SIAM),
  Philadelphia, PA, 2010, Insight and algorithms. \MR{2584074}

\bibitem{HLH80a}
G.~T. Herman, A.~Lent, and H.~Hurwitz, \emph{A storage-efficient algorithm for
  finding the regularized solution of a large, inconsistent system of
  equations}, J. Inst. Math. Appl. \textbf{25} (1980), no.~4, 361--366.
  \MR{578083}

\bibitem{HHLL79a}
Gabor~T. Herman, Henry Hurwitz, Arnold Lent, and Hsi~Ping Lung, \emph{On the
  {B}ayesian approach to image reconstruction}, Inform. and Control \textbf{42}
  (1979), no.~1, 60--71. \MR{538379}

\bibitem{johansson2009randomized}
Bj{\"o}rn Johansson, Maben Rabi, and Mikael Johansson, \emph{A randomized
  incremental subgradient method for distributed optimization in networked
  systems}, SIAM Journal on Optimization \textbf{20} (2009), no.~3, 1157--1170.

\bibitem{K-37}
Stefan Kaczmarz, \emph{Angen{\"a}herte {A}uflö{\"o}sung von {S}ystemen
  linearer {G}leichungen}, Bulletin International de l'Acad{\'e}mie Polonaise
  des Sciences et des Lettres. Classe des Sciences Math{\'e}matiques et
  Naturelles. S{\'e}rie A, Sciences Math{\'e}matiques (1937), 355--357.

\bibitem{kamath2015distributed}
Goutham Kamath, Paritosh Ramanan, and Wen-Zhan Song, \emph{Distributed
  randomized {K}aczmarz and applications to seismic imaging in sensor network},
  2015 International Conference on Distributed Computing in Sensor Systems, 06
  2015, pp.~169--178.

\bibitem{KwMy01}
Stanis{\l}aw Kwapie{\'n} and Jan Mycielski, \emph{On the {K}aczmarz algorithm
  of approximation in infinite-dimensional spaces}, Studia Math. \textbf{148}
  (2001), no.~1, 75--86. \MR{1881441 (2003a:60102)}

\bibitem{N-86}
Frank Natterer, \emph{The mathematics of computerized tomography}, Teubner,
  Stuttgart, 1986.

\bibitem{nedic2009distributed}
Angelia Nedic and Asuman Ozdaglar, \emph{Distributed subgradient methods for
  multi-agent optimization}, IEEE Transactions on Automatic Control \textbf{54}
  (2009), no.~1, 48.

\bibitem{NSW16a}
Deanna Needell, Nathan Srebro, and Rachel Ward, \emph{Stochastic gradient
  descent, weighted sampling, and the randomized {K}aczmarz algorithm}, Math.
  Program. \textbf{155} (2016), no.~1-2, Ser. A, 549--573. \MR{3439812}

\bibitem{NT14a}
Deanna Needell and Joel~A. Tropp, \emph{Paved with good intentions: analysis of
  a randomized block {K}aczmarz method}, Linear Algebra Appl. \textbf{441}
  (2014), 199--221. \MR{3134343}

\bibitem{NZZ15a}
Deanna Needell, Ran Zhao, and Anastasios Zouzias, \emph{Randomized block
  {K}aczmarz method with projection for solving least squares}, Linear Algebra
  Appl. \textbf{484} (2015), 322--343. \MR{3385065}

\bibitem{sayed2014adaptation}
Ali~H Sayed, \emph{Adaptation, learning, and optimization over networks},
  Foundations and Trends{\textregistered} in Machine Learning \textbf{7}
  (2014), no.~4-5, 311--801.

\bibitem{scaman2018optimal}
Kevin Scaman, Francis Bach, S{\'e}bastien Bubeck, Laurent Massouli{\'e}, and
  Yin~Tat Lee, \emph{Optimal algorithms for non-smooth distributed optimization
  in networks}, Advances in Neural Information Processing Systems, 2018,
  pp.~2740--2749.

\bibitem{Shah-2008}
Devavrat Shah, \emph{Gossip algorithms}, Foundations and
  Trends{\textregistered} in Networking \textbf{3} (2008), no.~1, 1--125.

\bibitem{SV-09a}
Thomas Strohmer and Roman Vershynin, \emph{\nic{a}{A} randomized {K}aczmarz
  algorithm with exponential convergence}, Journal of Fourier Analysis and
  Applications \textbf{15} (2009), no.~2, 262--278.

\bibitem{T-71}
Kunio Tanabe, \emph{Projection method for solving a singular system of linear
  equations and its application}, Numer. Math. \textbf{17} (1971), 203--214.

\bibitem{tsitsiklis1986distributed}
John Tsitsiklis, Dimitri Bertsekas, and Michael Athans, \emph{Distributed
  asynchronous deterministic and stochastic gradient optimization algorithms},
  IEEE Transactions on Automatic Control \textbf{31} (1986), no.~9, 803--812.

\bibitem{west1996graph}
Douglas~B. West, \emph{Introduction to graph theory}, Prentice Hall, Inc.,
  Upper Saddle River, NJ, 1996. \MR{1367739}

\bibitem{xiao2007distributed}
Lin Xiao, Stephen Boyd, and Seung-Jean Kim, \emph{Distributed average consensus
  with least-mean-square deviation}, Journal of Parallel and Distributed
  Computing \textbf{67} (2007), no.~1, 33--46.

\bibitem{Yos68a}
K\^{o}saku Yosida, \emph{Functional analysis}, Second edition. Die Grundlehren
  der mathematischen Wissenschaften, Band 123, Springer-Verlag New York Inc.,
  New York, 1968. \MR{0239384}

\bibitem{yuan2016convergence}
Kun Yuan, Qing Ling, and Wotao Yin, \emph{On the convergence of decentralized
  gradient descent}, SIAM Journal on Optimization \textbf{26} (2016), no.~3,
  1835--1854.

\bibitem{zhang2018compressed}
Xin Zhang, Jia Liu, Zhengyuan Zhu, and Elizabeth~S. Bentley, \emph{Compressed
  distributed gradient descent: Communication-efficient consensus over
  networks}, 2018, \texttt{arxiv.org/pdf/1812.04048}.

\end{thebibliography}

\newcommand{\nic}[1]{}
\providecommand{\bysame}{\leavevmode\hbox to3em{\hrulefill}\thinspace}
\providecommand{\MR}{\relax\ifhmode\unskip\space\fi MR }
\providecommand{\MRhref}[2]{%
  \href{http://www.ams.org/mathscinet-getitem?mr=#1}{#2}
}
\providecommand{\href}[2]{#2}



\end{document}